\documentclass[11pt]{amsart}
\usepackage{srcltx} 
\usepackage{latexsym} 
\usepackage[titletoc]{appendix}
\usepackage{lineno}
\usepackage{bm}
\usepackage{enumerate}
\usepackage{esint}
\usepackage{mathtools}
\usepackage[plainpages=true,pdfpagelabels,hypertexnames=true,colorlinks=true,pdfstartview=FitV,linkcolor=black,citecolor=black,urlcolor=black]{hyperref}
\usepackage[ddmmyyyy]{datetime}
\usepackage{etoolbox}
\usepackage{twoopt}
\usepackage{color}
\usepackage{amsfonts,amssymb,amscd,amsmath}

\DeclareMathOperator{\dv}{div}

\newcommand{\de}{\delta}
\newcommand{\ep}{\varepsilon}

\newcommand{\al}{\alpha}

\newcommand{\Om}{\Omega}

\newcommand{\RR}{\mathbb{R}}
\newcommand{\itl}[1][\Om]{\int_{#1}}
\def\aa{\mathcal{A}}

\newcommand{\V}{W_0^{1,p}(\Om)}

\newcommand{\axgrad}[1]{\mathcal{A}(x,\ifblank{#1}{\nabla u\:}{#1})}
\newcommand{\dx}[1][x]{\, d#1} 

\newcommand{\trm}[1]{\quad \textrm{#1}\quad }

\newcommand{\wpo}[1][\Om]{W_0^{1,p}(#1)}

\newcommand{\abs}[1]{\lvert #1 \rvert}
\newcommand{\expt}[1][\abs{T_j (u_n) }]{e^{\de #1}}
\newcommand{\exptk}[1][\abs{T_j (u_k) }]{e^{\de #1}}

\def\dv{\mathop{\rm div}}

\def\bea{\begin{equation}\begin{aligned}}
\def\ena{\end{aligned}\end{equation}}

\def\beas{\begin{equation*}\begin{aligned}}
\def\enas{\end{aligned}\end{equation*}}

\def\nc{\newcommand}

\nc\m[1]{\left| #1\right|}
\nc\norm[1]{\left\|#1\right\|}
\newtheorem{theorem}{Theorem}[section]
\newtheorem{lemma}[theorem]{Lemma}

\newtheorem{proposition}[theorem]{Proposition}

\newtheorem{definition}[theorem]{Definition}

\newtheorem{hypT}[theorem]{Hypothesis}

\newtheorem{assmT}[theorem]{Assumption}

\newtheorem{remark}[theorem]{Remark}        
\numberwithin{equation}{section}

\begin{document}
\title[Nonlinear equations with natural growth in the gradient]{Nonlinear equations with gradient natural growth  and 
distributional data, with applications to a  Schr\"odinger type equation}

\author[Karthik Adimurthi]
{Karthik Adimurthi$^{1}$}
\address{${}^1$ Department of Mathematical Sciences, 
Seoul National University, GwanAkRo 1, Gwanak-Gu, 
Seoul 08826, 
South Korea.}
\email{kadimurthi@snu.ac.kr \and karthikaditi@gmail.com}

\author[Nguyen Cong Phuc]
{Nguyen Cong Phuc$^{2}$}
\address{${}^2$ Department of Mathematics,
Louisiana State University,
303 Lockett Hall, Baton Rouge, LA 70803, USA.}
\email{pcnguyen@math.lsu.edu}

\thanks{$^{1}$ Supported in part by National Research Foundation of Korea grant funded by the Korean government (MEST) (NRF-2015R1A2A1A15053024)}
\thanks{$^{2}$ Supported in part by Simons Foundation, award number 426071}

\begin{abstract} We obtain necessary and sufficient conditions with sharp constants on the distribution $\sigma$
for the existence of a globally finite energy solution to the quasilinear equation with   
a gradient source term  of natural growth  of the form 
$-\Delta_p u = |\nabla u|^p + \sigma$ in a bounded open set $\Om\subset \RR^n$. Here $\Delta_p$, $p>1$, is the standard $p$-Laplacian operator defined by $\Delta_p u={\rm div}\, (|\nabla u|^{p-2}\nabla u)$. The class of solutions that we are interested in consists  of functions $u\in W^{1,p}_0(\Om)$ such that 
$e^{{\mu} u}\in W^{1,p}_0(\Om)$ for some ${\mu}>0$ and 
 the inequality 
\begin{equation*}
\int_{\Om} |\varphi|^p |\nabla u|^p dx  \leq A \int_\Om |\nabla \varphi|^p dx 
\end{equation*}
holds  for all  $\varphi\in C_c^\infty(\Omega)$ with some constant $A>0$.  This is a natural class of solutions at least when  the distribution $\sigma$ is nonnegative. The study of  $-\Delta_p u = |\nabla u|^p + \sigma$ is applied to show  the existence of  globally finite energy solutions to the quasilinear  equation of Schr\"odinger type $-\Delta_p v = \sigma\, v^{p-1}$, $v\geq 0$ in $\Om$, and $v=1$ on $\partial\Om$, via the exponential transformation $u\mapsto v=e^{\frac{u}{p-1}}$.

\end{abstract}
	
\maketitle

\section{Introduction}
The main goal of this paper is to address the solvability of  quasilinear elliptic 
equations with gradient nonlinearity of natural growth  of the form
\begin{equation}
 \label{basic_pde}
\left\{ \begin{array}{ll}
-\Delta_p u = |\nabla u|^p + \sigma & \trm{in} \Omega, \\
u = 0 & \trm{on} \partial \Omega, 
\end{array}
\right.
\end{equation}
{in} a bounded open set $\Omega \subset \RR^n$. Here  $\Delta_p u:= \dv (|\nabla u|^{p-2} \nabla u)$, $p>1$, is the $p$-Laplacian  and the datum $\sigma$ is a distribution in $\Om$.
More generally, we also consider the equation
\begin{equation}
 \label{basic_pde2}
\left\{ \begin{array}{ll}
-\dv \aa(x, u, \nabla u) = \mathcal{B}(x, u, \nabla u) + \sigma & \trm{in} \Omega, \\
u = 0 & \trm{on} \partial \Omega, 
\end{array}
\right.
\end{equation}
where the principal operator $\dv \aa(x, u, \nabla u)$ is a Leray-Lions operator defined on $W_0^{1,p}(\Om)$ and $|\mathcal{B}(x, u, \nabla u)| \lesssim |\nabla u|^p$.

The precise  assumptions on  the nonlinearities $\aa$, $B$ and the the precise definition of  solutions to   \eqref{basic_pde2} will be given in Section \ref{GenStru}.  Here we emphasize that  in this paper  we are interested only in 
{\it finite energy solutions $u$} with zero boundary condition in the sense that $u\in\wpo$. The energy space $\wpo$ is defined as the completion of $C_c^\infty(\Om)$ under the 
semi-norm $\norm{\nabla (\cdot)}_{L^{p}(\Om)}$. 
  
As an application of the study of \eqref{basic_pde}, we  also obtain existence of  finite energy solution to the quasilinear Schr\"odinger type equation
\begin{equation}\label{basic_pde-schr}
-\Delta_p v =  (p-1)^{1-p}\, \sigma\, v^{p-1}  \text{ in } \Omega, \qquad v \geq 0  \text{ in } \Omega, \qquad v = 1   \text{ on } \partial \Omega. 
\end{equation}

Equation \eqref{basic_pde} is a  prototype for quasilinear equations with natural growth in the gradient that has attracted a lot of attention in the past years. It can  be viewed as a quasilinear stationary version of a time-dependent viscous Hamilton-Jacobi equation, also known as the Kardar-Parisi-Zhang equation,  which appears  in the physical theory of surface growth  \cite{KPZ, KS}.

As far as existence is concerned, the nonlinearity $|\nabla u|^p$ in \eqref{basic_pde} is considered ``to have the bad sign'' and  by now it is well-known that in order for \eqref{basic_pde} to have a solution the datum $\sigma$ must be both {\it small and regular enough}. In particular, if $\sigma$ is a nonnegative distribution in $\Om$ (i.e., a nonnegative locally finite measure \textcolor{black}{in $\Om$}), then a necessary condition for the first equation in \eqref{basic_pde} to have a $W^{1,p}_{\rm loc}(\Om)$ solution is that (see \cite{HMV, JMV1, JMV2})
\begin{equation}\label{possig}
\int_{\Om} |\varphi|^p d\sigma  \leq \lambda \int_\Om |\nabla \varphi|^p dx \quad \text{for all } \varphi\in C_c^\infty(\Omega),
\end{equation}
with $\lambda=(p-1)^{p-1}$. Moreover, when $\sigma\geq 0$ the nonlinear term itself also obeys a similar Poincar\'e-Sobolev inequality
\begin{equation}\label{nablau-cond}
\int_{\Om} |\varphi|^p |\nabla u|^pdx  \leq A \int_\Om |\nabla \varphi|^p dx \quad \text{for all } \varphi\in C_c^\infty(\Omega),
\end{equation}
with $A=p^p$. 

Thus a  natural  space of solutions  associated to \eqref{basic_pde} is the space $\mathcal{S}$ of functions $u\in W^{1,p}_0(\Om)$ such that 
\eqref{nablau-cond} holds for some $A>0$. The main   
question  we wish to  address here is to find an optimal (largest) space $\mathcal{D}$ of `data' so that whenever $\sigma\in \mathcal{D}$ with sufficiently small norm $\norm{\sigma}_{\mathcal{D}}$  then \eqref{basic_pde} admits a  solution in $\mathcal{S}$.  In the case $\sigma\geq 0$ we can completely characterize the existence of finite energy solutions to  \eqref{basic_pde} in the following theorem. We remark again that in this case all $W^{1,p}_0(\Om)$ solutions automatically belong to ${\mathcal{S}}$ and \eqref{nablau-cond} holds with $A=p^p$.

\begin{theorem}\label{posmeas} Let $\sigma$ be a nonnegative locally finite measure in $\Om$. If  \eqref{basic_pde} has a solution in $u\in \wpo$  then 
 $\sigma\in (W^{1,p}_0(\Om))^*$ and \eqref{possig} holds with $\lambda=(p-1)^{p-1}$.  Conversely, if $\sigma\geq 0$, $\sigma\in (W^{1,p}_0(\Om))^*$, and  \eqref{possig} holds with $0<\lambda< (p-1)^{p-1}$ then \eqref{basic_pde} has a \emph{nonnegative} solution in $W^{1,p}_0(\Om)$ such that $e^{\frac{\delta u}{p-1}}-1\in W^{1,p}_0(\Om)$ for all $\delta\in [0, \delta_0)$ where $\delta_0=(p-1) \lambda^{\frac{-1}{p-1}}$.
\end{theorem}

In the linear case, $p=2$, these necessary and sufficient conditions have been observed in \cite{FV}. See also \cite{ADP} {(for $p=2$)} and \cite{HBV} for 
{certain} related results that were obtained by  different {methods}. We remark that, under a mild restriction on the domain,  by Hardy's inequality  (see \cite{Anc, Lew}), Theorem \ref{posmeas} covers the case of unbounded measure such as $\sigma =\varepsilon\, {\rm dist}(x, \partial\Om)^{-1}$ for some $\varepsilon>0$. It is also worth mentioning that
in the case $p=2$  and $\sigma$ is a nonnegative locally finite measure, other sharp existence results for  \eqref{basic_pde} {were} obtained in
\cite{HMV} for $\Omega=\RR^n$ and recently in \cite{FV} for bounded domains $\Omega$ with $C^2$ boundary under a very weak  notion of solution and boundary conditions.

The first part of Theorem \ref{posmeas} follows from the known necessary condition \eqref{possig}, H\"older's inequality, and 
the assumption that $\nabla u\in L^p(\Omega)$, since we have
$$\sigma=-|\nabla u|^p - {\rm div}\, (|\nabla u|^{p-2}\nabla u)\leq - {\rm div}\, (|\nabla u|^{p-2}\nabla u).$$ 

 On the other hand, the second part is a consequence of Theorem 
\ref{weakzero}  below that treats  even sign changing distribution  datum $\sigma$. This in fact is the main result that will be obtained in this paper. 
 
\begin{theorem}\label{weakzero} {\rm (i)}  Suppose that \eqref{basic_pde} has a solution in $u\in W^{1,p}_0(\Om)$ such that 
\eqref{nablau-cond} holds for some $A>0$ then  it necessarily holds that  $\sigma= {\rm div}\, (F) - |F|^{\frac{p}{p-1}}$
for a vector field  $F\in L^{\frac{p}{p-1}}(\Om,\RR^n)$ such that 
\begin{equation}\label{Fpower-cond}
 \int_{\Om} |F|^{\frac{p}{p-1}} |\varphi|^p dx \leq A \int_\Om|\nabla \varphi|^p dx \quad \text{for all } \varphi\in C_c^\infty(\Om). 
\end{equation}
In particular, both $\sigma$ and $|F|^{\frac{p}{p-1}}$ belong to the dual space  $(W^{1,p}_0(\Om))^*$.

\noindent {\rm (ii)} Conversely, suppose that  $\sigma={\rm div}\, F + f$ where $F\in L^{\frac{p}{p-1}}(\Om,\RR^n)$ and $f$ is a locally finite signed measure in $\Om$ with $|f|\in (W^{1,p}_0(\Om))^*$ such that 
\begin{equation}\label{datasmallness}
p \int_{\Om} |F| |\varphi|^{p-1} |\nabla \varphi|dx +\int_{\Om} |\varphi|^p d|f| \leq \lambda  \int_\Om|\nabla \varphi|^p dx \qquad \forall   \varphi\in C_c^\infty(\Om),
\end{equation}
 for some   $\lambda\in (0, (p-1)^{p-1})$.
Then equation  \eqref{basic_pde} has a (possibly sign changing) solution  $u\in W^{1,p}_{0}(\Om)$ such that $e^{\frac{\delta u}{p-1}}-1\in W^{1,p}_0(\Om)$ for all $\delta\in [0, \delta_0)$ where $\delta_0=(p-1) \lambda^{\frac{-1}{p-1}}$. This solution  satisfies the Poincar\'e-Sobolev inequality \eqref{nablau-cond} for some $A=A(p)>0$. 
{Moreover, if $\lambda\in (0, (p-1)^{{\rm min}\{1, p-1\}})$, then   both $e^{\frac{u}{p-1}}-1$ and $e^{u}-1$ belong to  $W^{1,p}_0(\Om)$.}
\end{theorem}


Several remarks regarding Theorem \ref{weakzero} are now in order.
\begin{remark} By approximation and Fatou's lemma, inequalities \eqref{Fpower-cond} and \eqref{datasmallness} actually hold for all $\varphi\in W^{1,p}_0(\Om)$. The integral 
$\int_{\Om} |\varphi|^p d|f|$ makes sense even for $\varphi\in W^{1,p}_0(\Om)$ since $|f|$ is continuous with respect to the capacity ${\rm cap}_p(\cdot,\Om)$ and $\varphi$ has a 
${\rm cap}_p$-quasicontinuous representative, whose values are defined  ${\rm cap}_p$-quasieverywhere in $\Om$.
  Here ${\rm cap}_p(\cdot,\Om)$ is the \textcolor{black}{variational} $p$-capacity associated to $\Om$ defined for each compact set $K\subset\Om$ by
$${\rm cap}_{p}(K, \Om):=\inf\left\{\int_{\Om} |\nabla \phi|^p dx:\phi\in C_c^\infty(\Om) {\rm ~and~} \phi\geq \chi_K  \right\}.$$
\end{remark}

\begin{remark} By H\"older's inequality we see that if $F$ satisfies \eqref{Fpower-cond} for some $A>0$ then 
\begin{equation*}
p \int_{\Om} |F| |\varphi|^{p-1} |\nabla \varphi|dx  \leq p A^{\frac{p-1}{p}}  \int_\Om|\nabla \varphi|^p dx \qquad \forall   \varphi\in C_c^\infty(\Om).
\end{equation*}
Thus by Theorem \ref{weakzero}(ii) if $F\in L^{\frac{p}{p-1}}(\Om,\RR^n)$ satisfies \eqref{Fpower-cond} for some $0<A< (p-1)^p p^{-\frac{p}{p-1}}$ then the equation 
$-\Delta_p u =|\nabla u|^p +{\rm div}\, F$ \textcolor{black}{has a solution in  $W^{1,p}_0(\Om)$}.
\end{remark}

\begin{remark} Let $\mu$ be a nonnegative locally finite measure in $\Om$.  It is well-known that the inequality 
\begin{equation*}
\int_{\Om} |\varphi|^p d\mu \leq A_1 \int_\Om|\nabla \varphi|^p dx \qquad \forall   \varphi\in C_c^\infty(\Om)
\end{equation*}
is equivalent to the condition
\begin{equation}\label{Cap-A2}
\mu(K)\leq A_2 \, {\rm cap}_{p}(K, \Om)
\end{equation}
for all compact sets $K\subset\Om$ (see \cite[Chapter 2]{Maz}). 

Thus in \textcolor{black}{$(ii)$ of Theorem   \ref{weakzero}} , condition \eqref{datasmallness} can be replaced by \eqref{Cap-A2}
with $\mu=|F|^{\frac{p}{p-1}}+|f|$ for a sufficiently small constant $A_2>0$.  

\textcolor{black}{Moreover, by $(ii)$ of Theorem   \ref{weakzero}}, if  $f$ is a locally finite signed measure in $\Om$ with $|f|\in (W^{1,p}_0(\Om))^*$ such that
\eqref{Cap-A2} holds with $d\mu=d|f|$, the we have a decomposition 
$$f={\rm div}\, F -g,$$ where \textcolor{black}{$F\in L^{\frac{p}{p-1}}(\Om,\RR^n)$} and $g\in L^1(\Om), \, g\geq0,$ such that \textcolor{black}{the} $L^1$ function  $\mu:=(|F|^{\frac{p}{p-1}}+g)$ also satisfies \eqref{Cap-A2}. See  \cite{BGO, FS} for a similar decomposition of measures that are  continuous w.r.t the $p$-capacity.
\end{remark}

\begin{remark} Let $L^{s,\infty}(\Om)$, $s\geq 1$, denote the weak $L^s$ space on $\Om$ with quasinorm
$$\norm{g}_{L^{s,\infty}(\Om)}:= \sup_{t>0}t |\{ x\in\Om: |g(x)|>t\}|^{1/s}.$$
\textcolor{black}{For $g\in L^{\frac{n}{p}, \infty}(\Om)$ with $1<p<n$,} it is known that \textcolor{black}{(see, e.g., \cite[Eqn. (2.6)]{FM3})}
$$\int_\Om |\varphi|^p g dx \leq S_{n,p} \textcolor{black}{\norm{g}_{L^{\frac{n}{p},\infty}(\Om)}} \int_\Om |\nabla \varphi|^p dx \qquad \forall   \varphi\in C_c^\infty(\Om),$$ 
where the constant $S_{n,p}$ is given by
$$S_{n,p}= \left[ \frac{p}{\sqrt{\pi}(n-p)}  \right]^p  \Gamma(1+n/2)^{p/n}.$$

This shows that in Theorem   \ref{weakzero}(ii),  condition \eqref{datasmallness} can be replaced by 
\textcolor{black}{$|F|^{\frac{p}{p-1}}+|f|\in L^{\frac{n}{p},\infty}(\Om)$} with a sufficiently small norm. Existence results under this weak norm 
condition have been obtained in \cite{FM3}. See also the earlier works \cite{FM1, FM2} where the strong norm condition involving \textcolor{black}{$L^{\frac{n}{p}}(\Om)$} was used instead. 
More general existence results in which $|F|^{\frac{p}{p-1}}+|f|$ is assumed to be small in the norm of certain Morrey spaces can  be found in the recent paper \cite{MP}. 
Those Morrey space conditions are  also stronger than condition  \eqref{datasmallness} as they fall into the realm of Fefferman-Phong type conditions (see, e.g.,  \cite{DPT, ChWW, Fef, P, SW}).
\end{remark}

We now discuss the Schr\"odinger  type equation with distributional potential \eqref{basic_pde-schr}. This equation  is interesting in its own right and has a  strong connection to equation \eqref{basic_pde}
as being observed and exploited,  e.g., in \cite{ADP, HBV, JMV1, JMV2}.

\textcolor{black}{ By a solution to \eqref{basic_pde-schr}, we mean the following definition.}

\textcolor{black}{\begin{definition} Let $\sigma\in (W^{1,p}_0(\Om))^*$. A function $v$ defined in $\Om$ is a solution of  \eqref{basic_pde-schr} if  $v\geq 0$,   $v-1\in W^{1,p}_{0}(\Om)$,  $v^{p-1}\in W^{1,p}_{\rm loc}(\Om)$, and 
\begin{equation}\label{weak-schr}
\int_{\Omega} |\nabla v|^{p-2} \nabla v \cdot \nabla \varphi dx= (p-1)^{1-p} \langle \sigma, v^{p-1}\varphi\rangle \qquad \forall  \varphi\in C_c^\infty(\Om). 
\end{equation}
\end{definition}}

\textcolor{black}{Note that the right hand side of \eqref{weak-schr} makes sense since $v^{p-1}\varphi\in W^{1,p}_0(\Om)$ and 
$\sigma\in (W^{1,p}_0(\Om))^*$.}

Formally, by making the change of unknowns $v=e^{\frac{u}{p-1}}$, equation \eqref{basic_pde} is transformed into the
 Schr\"odinger  type equation \eqref{basic_pde-schr}.  Indeed, it is possible to show rigorously  that Theorem \ref{weakzero} implies the existence of finite energy solutions 
to  \eqref{basic_pde-schr}:

\begin{theorem}\label{Schro-type} Suppose that  $\sigma={\rm div}\, F + f$ where \textcolor{black}{$F\in L^{\frac{p}{p-1}}(\Om,\RR^n)$} and $f$ is a locally finite signed measure in $\Om$ with $|f|\in (W^{1,p}_0(\Om))^*$ such that 
\begin{equation*}
p \int_{\Om} |F| |\varphi|^{p-1} |\nabla \varphi|dx +\int_{\Om} |\varphi|^p d|f| \leq \lambda  \int_\Om|\nabla \varphi|^p dx 
\qquad \forall  \varphi\in C_c^\infty(\Om),
\end{equation*}
 for some   $\lambda\in (0, (p-1)^{\min\{1, p-1\}})$.
{Then equation  \eqref{basic_pde-schr} has a nonnegative solution $v$ such that both  \textcolor{black}{$v-1$ and 
$v^{p-1}-1$} belong to  \textcolor{black}{$W^{1,p}_{0}(\Om)$}. Moreover,  $v$   satisfies the following Poincar\'e-Sobolev inequality
\begin{equation}\label{WNforv}
 \int_{\Om} \left|\frac{\nabla v}{v}\right|^p |\varphi|^p dx \leq A \int_\Om|\nabla \varphi|^p dx \qquad \forall \varphi\in C_c^\infty(\Om), 
\end{equation}
with a constant $A=A(p)>0$.}
\end{theorem}

\begin{remark}  If the factor $(p-1)^{1-p}$ on the right-hand side of \eqref{basic_pde-schr} is dropped  then the smallness condition on  
$\lambda$ becomes $\lambda\in (0, p^{\#})$, where $p^{\#}=(p-1)^{2-p}$ if $p> 2$ and $p^{\#}=1$ if $p\leq 2$ as in 
\cite{JMV2}. The sharpness of $p^{\#}$ (and \textcolor{black}{that} of $(p-1)^{\min\{1, p-1\}}$ for \eqref{basic_pde-schr}) was also justified in \cite{JMV2}.
\end{remark}

{One could also treat the Schr\"odinger type equation \eqref{basic_pde-schr} in a more general fashion, where the standard $p$-Laplacian is replaced by a quasilinear elliptic operator with merely measurable `coefficients'. See Remark \ref{mea-coeff} below and see also  \cite{JMV2}.}

We mention that the existence of  finite energy solutions to \eqref{basic_pde-schr}  in the case $\sigma \geq 0$ was obtained in \cite{HBV}   by {a  method that does not seem to work for  sign changing $\sigma$} (see also \cite{ADP} for  $p=2$).    On the other hand, the  work \cite{JMV2} (see also \cite{JMV1}) obtains a locally finite energy solution  $v\in W^{1,p}_{\rm loc}(\Om)$  to the first two equations in  \eqref{basic_pde-schr} \textcolor{black}{(\emph{without any boundary conditions})} only under the mild restriction
$$-\Lambda \int_\Om |\nabla \varphi|^p dx \leq \langle \sigma, |\varphi|^p \rangle \leq \lambda \int_\Om |\nabla \varphi|^p dx \quad \text{for all } \varphi\in C_c^\infty(\Om)$$
for some  $\lambda\in (0, (p-1)^{\min\{1, p-1\}})$ and $\Lambda\in (0,+\infty)$. Moreover, $v$ also satisfies  \eqref{WNforv} for some  $A>0$.  Then, also under the  restriction $\lambda\in (0, (p-1)^{\min\{1, p-1\}})$, by the logarithmic transformation
$u=(p-1)\log(v)$ it was  obtained in \cite{JMV2}, a solution $u\in W^{1,p}_{\rm loc}(\Om)$ to the first equation in \eqref{basic_pde} (but without any boundary condition) that also satisfies \eqref{nablau-cond} for some $A>0$.

\emph{In this paper, we follow an opposite route, i.e.,  we first treat equation \eqref{basic_pde} directly and then deduce existence for the Schr\"odinger type equation \eqref{basic_pde-schr} from it.} This way, we are able to treat equation \eqref{basic_pde} in its most general form, i.e.,  the nonlinear equation with general structure \eqref{basic_pde2}.   Moreover, for equation \eqref{basic_pde} we obtain larger upper bound  for $\lambda$ in the existence condition \eqref{datasmallness} $( {\rm i.e.,~} (p-1)^{p-1}$ versus $(p-1)^{\min\{ 1, p-1\}})$. Our approach to \eqref{basic_pde2} is a refinement of the approach of  V. Ferone and F. Murat in \cite{FM2, FM3}. The main difficulties to overcome here are  the generality nature of $\sigma$ and the sharpness of the smallness constants. In particular, in this scenario one does not gain any higher integrability on the nonlinear term  $\mathcal{B}(x,u, \nabla u)$, which makes it impossible to follow  a  compactness argument as in \cite{MP}.
Moreover, in order for us to apply the existence results of \eqref{basic_pde} to \eqref{basic_pde-schr} we need to find a solution $u$  of \eqref{basic_pde} with the additional property that both  $e^{\frac{u}{p-1}}-1$ and   $e^{u}-1$ belong to $W^{1,p}_0(\Om)$ as stated in Theorem \ref{weakzero}.

\section{Equations with general nonlinear structure}\label{GenStru}
As we have mentioned, existence results in the spirit of Theorem \ref{weakzero}(ii) also hold for  equations with a more general nonlinear structure 
\eqref{basic_pde2}. For that we need the following assumptions on the nonlinearities  $\mathcal{A}$ and $\mathcal{B}$:  

\noindent{\bf Assumption on $\mathcal{A}$.} The nonlinearity $\aa : \Om \times \RR\times \RR^n \rightarrow \RR^n$ is a  Carath\'edory  function, i.e., $\aa(x, s, \xi)$ is measurable in $x$ for every $(s,\xi)$ and continuous in $(s,\xi)$ for a.e. $x\in\Om$. For some $p>1$, it holds that
\begin{gather}
\label{monotone-strict} \langle \aa(x, s, \xi) - \aa(x, s, \eta), \xi - \eta \rangle >0,\\
\label{coercivity}  \langle \aa(x, s, \xi), \xi\rangle \geq \alpha_0 |\xi|^p,\\
\label{growth-p} |\aa(x, s, \xi)| \leq a_0|\xi|^{p-1} + a_1 |s|^{p-1}
\end{gather}
for every  $(\xi, \eta)\in \RR^n \times \RR^n$, $\xi\not=\eta$, and a.e. $x \in\Om$. Here $\alpha_0>0$, and  $a_0, a_1\geq 0$.

\noindent{\bf Assumption on $\mathcal{B}$.} The nonlinearity $\mathcal{B} : \Om \times \RR \times \RR^n \rightarrow \RR$ is a Carath\'edory  function which satisfies, for a.e. $x\in\Om$, every $s\in\RR$, and every $\xi\in\RR^n$,
\begin{equation}\label{Bcond}
|\mathcal{B}(x,s,\xi)|\leq b_0 |\xi|^p +b_1 |s|^m, \quad \mathcal{B}(x,s,\xi){\rm sign}(s) \leq \alpha_0 \gamma_0 |\xi|^p,
\end{equation}
 where $m>0$, and $b_0, b_1$, $\gamma_0\geq 0$.  Here $\alpha_0$ is as given in \eqref{coercivity}.

By a solution of  \eqref{basic_pde2} we mean  the following. 
\begin{definition} Under \eqref{monotone-strict}-\eqref{Bcond}, a function $u \in W^{1,p}_{0}(\Om)$ is a solution of  \eqref{basic_pde2} if   $\mathcal{B}(x, u, \nabla u)\in L^1_{\rm loc}(\Om)$ and 
\begin{equation*} 
 \itl \aa(x, u, \nabla u) \cdot \nabla \varphi \ dx = \itl \mathcal{B}(x, u, \nabla u) \varphi \ dx + \langle\sigma, \varphi\rangle
\end{equation*}
\textcolor{black}{holds} for all test functions $\varphi \in C_c^{\infty} (\Omega)$. 
\end{definition}

We remark that in the case $\mathcal{B}(x, u, \nabla u)\in L^1(\Om)$ and $\sigma\in (W^{1,p}_{0}(\Om))^{*}$,   we can take any function $\varphi\in W^{1,p}_{0}(\Om)\cap L^\infty(\Om)$ as a test function in the above definition. This  follows from a result of Br\'ezis and Browder \cite{BB} as we have $\mathcal{B}(x, u, \nabla u)\in(W^{1,p}_{0}(\Om))^{*} \cap  L^1(\Om)$. It can also be seen by  approximating $\varphi$ in $W^{1,p}_0(\Om)$ by a sequence $\varphi_j\in C_c^\infty(\Om)$ such that $|\varphi_j| \leq |\varphi|\leq M$ a.e. (using Theorem 9.3.1 in \cite{AH} and suitable convolutions). 

We mention that in the special case $|\mathcal{B}(x, u, \nabla u)| \in (W^{1,p}_{0}(\Om))^{*} \cap  L^1(\Om)$\textcolor{black}{, we} can even drop the condition $\varphi\in L^\infty(\Om)$. In fact, we have the following more general result. 

\begin{lemma}\label{dis-mea} Suppose that  $f$ is a locally finite signed measure in $\Om$ with $|f|\in (W^{1,p}_0(\Om))^*$. Then for any $\varphi\in W^{1,p}_0(\Om)$ we have
$$\langle f, \varphi\rangle= \int_{\Om} \widetilde{\varphi}\, d f,$$
where $\widetilde{\varphi}$ is any ${\rm cap}_p$-quasicontinuous representative of $\varphi$. 
\end{lemma}

In the case $f$ is nonnegative, the proof of Lemma \ref{dis-mea} can be found in \cite[Lemma 2.5]{Mik}. The general case also follows from \textcolor{black}{that,} since $f=f^{+} + f^{-}$ and both $f^{+}$ and   $f^{-}$  belong to   $(W^{1,p}_0(\Om))^*$. In what follows, when dealing with pointwise behavior of functions in $W^{1,p}_0(\Om)$ we will implicitly use their ${\rm cap}_p$-quasicontinuous representatives. Lemma \ref{dis-mea} will be \textcolor{black}{used,} e.g., in \eqref{realize-sig} below.

Under the above assumptions on $\mathcal{A}$ and $\mathcal{B}$, we obtain the following existence result.

\begin{theorem}\label{MainExistence} Let $\sigma={\rm div}\, F + f$ where \textcolor{black}{$F\in L^{\frac{p}{p-1}}(\Om,\RR^n)$} and $f$ is a locally finite signed measure in $\Om$ with $|f|\in (W^{1,p}_0(\Om))^*$ such that 
\begin{equation}\label{smalllambda}
p \int_{\Om} |F| |\varphi|^{p-1} |\nabla \varphi|dx +\int_{\Om} |\varphi|^p d|f| \leq \lambda  \int_\Om|\nabla \varphi|^p dx  
\end{equation}
holds for all $\varphi\in C_c^\infty(\Om)$, with   $\lambda\in (0, \gamma_0^{1-p}\alpha_0(p-1)^{p-1}).$
  Then there exists a  solution $u\in W_0^{1,p}(\Om)$ to  the equation 
	\begin{equation}\label{basic_pde3}
-\dv \aa(x, u, \nabla u) = \mathcal{B}(x, u, \nabla u) + \sigma  \trm{in} \Omega,
\end{equation}
such that $e^{\frac{\delta|u|}{p-1}}-1\in W^{1,p}_0(\Om)$  for all $\de\in [\gamma_0, \de_0)$,  with \textcolor{black}{$\delta_0=(p-1) \left(\frac{\alpha_{0}}{\lambda}\right)^{\frac{1}{p-1}}$.}

{ Moreover,  for any $\delta_1 > \gamma_0$ such that   \eqref{smalllambda} holds with
\begin{equation}\label{lambdacond2}
\lambda<\left(\frac{p-1}{\delta_1}\right)^p \alpha_0 \left(\frac{\delta_1}{p-1} +\delta_1-\gamma_0\right),
\end{equation}
 we have  $e^{\frac{\delta_1|u|}{p-1}}-1 \in {\V}$.} 
\end{theorem}

{\begin{remark}  It is easy to check that, for $\delta_1>\gamma_0$ one has
\begin{equation*}
\left(\frac{p-1}{\delta_1}\right)^p \alpha_0 \left(\frac{\delta_1}{p-1} +\delta_1- \gamma_0\right) < \gamma_0^{1-p}\alpha_0(p-1)^{p-1}.
\end{equation*}
Moreover, \textcolor{black}{for example with} $p>2$ and $\alpha_0=\gamma_0=1$\textcolor{black}{,  if}  \eqref{smalllambda} holds with $\lambda<p-1 \in (0, (p-1)^{p-1})$\textcolor{black}{, then} we see that
\eqref{apri1} holds with $1\leq \delta< (p-1)(p-1)^{\frac{-1}{p-1}}$, but it does not allow us to take \textcolor{black}{$\delta=p-1$!} On the other hand, for $\lambda<p-1$ inequality  
\eqref{lambdacond2} holds with $\de_1=p-1$ and thus $e^{|u|}-1 \in {\V}$.
\end{remark}}

Due to the general structures of $\mathcal{A}$ and $\mathcal{B}$, here we do not claim that the solution $u$
 obtained in Theorem \ref{MainExistence} satisfies the Poincar\'e-Sobolev inequality \eqref{nablau-cond}.

The paper is organized as \textcolor{black}{follows:} In Section \ref{ActualExistence}, we provide the proof of Theorem \ref{MainExistence}.  This proof is based on  the existence of solutions to an {approximate} equation along with certain uniform bounds given in Section \ref{ApproxExist}.  These important uniform bounds are in turn deduced
from the a priori estimate of Section \ref{AprioriEst}, though not directly. Finally, the proof of Theorems \ref{weakzero} and \ref{Schro-type} will be given in Section \ref{main-proofs}.

\section{An a priori estimate}\label{AprioriEst}

In this section, we obtain certain  exponential type a priori bounds for solutions  of 
\begin{equation}\label{basic-ep}
-\dv \aa(x, u, \nabla u) +\varepsilon\,  |u|^{p-2} u= \mathcal{B}(x, u, \nabla u) + \sigma  \trm{in} \Omega,
\end{equation}
where $\varepsilon\geq0$. The case $\varepsilon>0$ will be needed in the next section to absorb certain unfavorable terms in the approximating process; see \eqref{ep-absorb} below. Earlier, this idea was 
implemented by V. Ferone and F. Murat   in \cite{FM3}.  
\begin{theorem}
 \label{regularity_Murat-Ferone} 
Let $\sigma={\rm div}\, F + f$ where \textcolor{black}{$F\in L^{\frac{p}{p-1}}(\Om,\RR^n)$} and $f$ is a locally finite signed measure in $\Om$ with $|f|\in (W^{1,p}_0(\Om))^*$ such that 
\eqref{smalllambda} holds
for all $\varphi\in C_c^\infty(\Om)$, with   $\lambda\in (0, \gamma_0^{1-p}\alpha_0(p-1)^{p-1}).$
Then for any  $\varepsilon \geq 0$ and any $W_0^{1,p}(\Om)$ solution $u$ to  equation \eqref{basic-ep}
such that $e^{\frac{\delta|u|}{p-1}}-1\in W^{1,p}_0(\Om),$
we have 
\begin{equation}\label{apri1}
 \|u\|_{\V} + \|e^{\frac{\delta|u|}{p-1}}-1\|_{\V} \leq M_{\de}. 
\end{equation}
provided $\de\in [\gamma_0, \de_0)$ where  $\delta_0=(p-1) (\alpha_{0}/\lambda)^{\frac{1}{p-1}}$. Here  $M_\delta$ is independent of $u$ and $\varepsilon$.

{Moreover, for any $\delta_1 > \gamma_0$ such that $e^{\frac{\delta_1|u|}{p-1}}-1\in W^{1,p}_0(\Om)$, and \eqref{smalllambda} holds with $\lambda$ satisfying \eqref{lambdacond2}, we have}
\begin{equation}\label{apri2}
 \|e^{\frac{\delta_1|u|}{p-1}}-1\|_{\V} \leq M_{\de_{1}} + C_{\de_{1}} \norm{\nabla u}_{L^p(\Om)}. 
\end{equation}
The constants $M_{\de_{1}}$  and $C_{\de_{1}}$ are independent of $u$ and $\varepsilon$. 
\end{theorem}

\begin{proof} Let  $u\in \wpo$ be a solution of \eqref{basic-ep} and define
$$w={\rm sign}(u)[e^{\mu|u|}-1]/\mu, \qquad {\rm with~} \mu=\delta/(p-1),$$
where ${\rm sign}(u)=0$ if $u=0$, ${\rm sign}(u)=1$ if $u>0$, and ${\rm sign}(u)=-1$ if $u<0$.
Then from the assumption $e^{\mu|u|}-1\in \wpo$, we see that $w\in \wpo$ with
\begin{equation}\label{deriofw}
\nabla w=e^{\mu|u|} \nabla u.
\end{equation}

Indeed, for $\varepsilon>0$ define $f_\varepsilon(x)= \frac{x}{\sqrt{x^2 +\varepsilon^2}}$, $x\in\RR$, {and
denote by $T_s$, $s>0$, the two-sided truncation operator at level $s$, i.e., 
\textcolor{black}{\begin{equation*}
T_s(r)=r \ {\rm ~if~} |r|\leq s \quad {\rm and} \quad  T_s(r)= {\rm sign}(r) s \ {\rm ~if~} |r|>s,
\end{equation*}
then it follows} that $T_s(u)\in W^{1,p}_0(\Om)$ for any $s>0$ and
\begin{eqnarray*}
\nabla\left[ f_\varepsilon(T_s(u)) (e^{\mu|T_s(u)|}-1)/\mu\right]&=&\frac{\nabla T_s(u) \ep^2}
{(T_s(u)^2  +  \ep^2)^{3/2}} (e^{\mu|T_s(u)|}-1)/\mu \\
&& +\, f_\ep(T_s(u)) \nabla (e^{\mu|T_s(u)|}-1)/\mu
\end{eqnarray*}
in the weak sense. Note that 
$$\varepsilon^2 (e^{\mu|T_s(u)|}-1)/\mu\leq \frac{e^{\mu s}}{\mu s} \, \varepsilon^2 |T_s(u)|\leq  \frac{e^{\mu s}}{\mu s} \, (|T_s(u)|^2  +  \ep^2)^{3/2},$$
$$f_\varepsilon(T_s(u))\rightarrow {\rm sign}(T_s(u))={\rm sign}(u) \text{ as } \varepsilon\rightarrow 0^{+},$$
and thus by Dominated Convergence Theorem we find
$$\nabla\left[ {\rm sign}(u) (e^{\mu|T_s(u)|}-1)/\mu\right]={\rm sign}(u) \nabla (e^{\mu|T_s(u)|}-1)/\mu=e^{\mu|T_s(u)|}\nabla T_s(u).$$
 \textcolor{black}{Now} using the assumption $e^{\mu|u|}|\nabla u|\in L^p(\Om)$ and letting $s\rightarrow \infty$, we obtain \eqref{deriofw}.}

 For each  $s>0$, we will use the following test function for  \eqref{basic-ep}:
$$v_s=e^{\de |u_s|} w_s,$$
where $u_s= T_s(u)$ and $w_s={\rm sign}(u) [e^{\mu |u_s|} -1]/\mu$ with $\mu=\delta/(p-1)$. 

From the definition of $w_s$ we have $|w_s| \leq |w|$ and $\nabla w_s =e^{\mu |u_s|}\nabla u_s$.
 Thus both $w_s$ and $v_s$   belong to $\wpo\cap L^\infty(\Om)$ and  moreover, 
$$\nabla v_s = \Big[e^{\de |u|} \nabla w + \de |w| e^{\de |u|}  \nabla u \Big]\chi_{\{|u|\leq s\}}.$$ 

Using $v_s$ as a test function in \eqref{basic-ep},  we get
\begin{eqnarray*}
 \lefteqn{\itl \aa(x, u, \nabla u) \cdot \nabla w e^{\de |u|} \chi_{\{|u|\leq s\}} \ dx + \varepsilon \itl |u|^{p-2} u e^{\delta |u_s|} w_s \ dx}\\
  &=&- \itl \de |w| e^{\de |u|} \aa(x, u, \nabla u)\cdot \nabla u \chi_{\{|u|\leq s\}} \ dx +\\
&&  + \ \itl \mathcal{B}(x, u,  \nabla u) e^{\de |u_s|} w_s \ dx + \itl F \cdot \nabla v_s  \ dx + \itl v_s df. 
\end{eqnarray*}

We now write this equality as
\begin{equation}\label{I12345}
I_1 + I_2  = I_3 + I_4 + I_5 +I_6,
\end{equation}
where $I_i$, $i\in\{1, \dots ,6\}$, are the corresponding terms.

\noindent {\bf Estimate for $I_1$:} Since  $\nabla w_s = e^{\mu |u_s|} \nabla u_s$, using the coercivity condition \eqref{coercivity}, we see that
 \begin{eqnarray} \label{I1}
  I_1 & =& \itl \aa(x, u_s, \nabla u_s) \cdot \nabla w_s e^{\de |u_s|} \chi_{\{|u|\leq s\}}  \ dx \\
& =& \itl \aa(x, u_s, \nabla u_s) \cdot \nabla u_s e^{(\mu+\de) |u_s|} \ dx \nonumber\\
& \geq& \alpha_0 \itl |\nabla w_s|^p \ dx, \nonumber
 \end{eqnarray}
where we  used the fact  $\mu + \de = p\mu$.

\noindent {\bf Estimate for $I_2$:} We have 
\begin{equation}\label{I1prime}
 I_2= \varepsilon \itl |u|^{p-1}  e^{\delta |u_s|} \frac{e^{\mu |u_s|}-1}{\mu}   \ dx \geq \varepsilon \,  s^{p-1} \itl e^{\delta s} \frac{e^{\mu s}-1}{\mu} \chi_{\{|u|>s\}}\geq 0.
\end{equation}

\noindent {\bf Estimate for $I_3 + I_4$:} By \eqref{coercivity} we have 
 \begin{eqnarray*}
  I_3 + I_4 &=& - \itl \de |w| e^{\de |u|} \aa(x, u, \nabla u)\cdot \nabla u \chi_{\{|u|\leq s\}} dx + \itl \mathcal{B}(x, u, \nabla u) e^{\de |u_s|} w_s dx \\
&\leq& - \itl \de \al_0 |w_s| e^{\de |u_s|}  |\nabla u_s|^p \ dx + \itl \mathcal{B}(x, u,  \nabla u) {\rm sign}(u) e^{\de |u_s|} |w_s| dx \\	
&=& - \itl \de \al_0 |w_s| e^{\de |u_s|}  |\nabla u_s|^p \ dx + \\
&& +\ \itl \mathcal{B}(x, u, \nabla u) {\rm sign}(u) e^{\de |u_s|} |w_s| [\chi_{\{|u|\leq s\}} + \chi_{\{|u|> s\}}]  dx.
 \end{eqnarray*}

Since we assume  $\de \geq \gamma_0$, this and the second condition in \eqref{Bcond}  imply that 
\begin{eqnarray}\label{I34B}
 I_3 + I_4 &\leq &  \itl (-\de+\gamma_0) \al_0 |w_s| e^{\de |u_s|}  |\nabla u_s|^p \ dx \\
&& +\ \itl \mathcal{B}(x, u, \nabla u) {\rm sign}(u) e^{\de |u_s|} |w_s| \chi_{\{|u|> s\}}  dx\nonumber\\
&\leq& \itl \mathcal{B}(x, u, \nabla u) {\rm sign}(u) e^{\de |u_s|} |w_s| \chi_{\{|u|> s\}}  dx.\nonumber
\end{eqnarray}

Thus by the first inequality on \eqref{Bcond} and the fact that 
\begin{eqnarray}\label{exptopo}
|v_s|&=& e^{\de |u_s|} |w_s|=(1+\mu |w_s|)^{\delta/\mu} |w_s|=(1+\mu |w_s|)^{p-1} |w_s|\\
 &\leq& \frac{1}{\mu}(1+\mu |w_s|)^{p}\leq \frac{1}{\mu} e^{p\mu |u_s|}\leq \frac{1}{\mu} e^{p\mu |u|},\nonumber
\end{eqnarray}
we find 
\begin{eqnarray}\label{I2I3}
I_3 + I_4 &\leq&  \itl (b_0 |\nabla u|^p + b_1 |u|^{m}) \frac{1}{\mu}e^{p\mu |u|} \chi_{\{|u|> s\}}  dx\\
&=&  \frac{1}{\mu}\itl b_0 |\nabla w|^p \chi_{\{|u|> s\}}  dx   + \frac{1}{\mu}\itl b_1 |u|^{m} e^{p\mu |u|} \chi_{\{|u|> s\}}  dx. \nonumber
\end{eqnarray}

\noindent {\bf Estimate for $I_5 +I_6$:} Using \eqref{exptopo}  again and Lemma \ref{dis-mea} we  have
\begin{eqnarray}\label{realize-sig}
 I_5 +I_6 &=& \itl F \cdot \nabla[(1+\mu |w_s|)^{p-1} w_s] \ dx + \itl v_s d f\\
&=&  \itl F \cdot [(p-1)(1+\mu |w_s|)^{p-2}\nabla w_s \, {\rm sign}(w_s)\, \mu  w_s] dx+ \nonumber\\
&& +\, \itl F \cdot [(1+\mu |w_s|)^{p-1} \nabla w_s] dx + \itl v_s df\nonumber\\
&\leq & p\itl |F|  (1+\mu |w_s|)^{p-1} |\nabla w_s| dx  +  \itl (1+\mu |w_s|)^{p-1}|w_s| d|f|.\nonumber
\end{eqnarray}

Using the inequality
$$(1+\mu |w_s|)^{p-1}\leq (1+{\color{black}\tilde{\varepsilon}}) \mu^{p-1}|w_s|^{p-1} + C({\color{black}\tilde{\varepsilon}}, p),\qquad {\color{black}\tilde{\varepsilon}}>0,$$ 
 and H\"older's inequality we have

\begin{eqnarray*}
I_5+ I_6&\leq& (1+{\color{black}\tilde{\varepsilon}})  \mu^{p-1} \, p \itl |F|   |w_s|^{p-1} |\nabla w_s|dx  + (1+{\color{black}\tilde{\varepsilon}}) \mu^{p-1} \itl |w_s|^p d|f| \\
&&+ \   C({\color{black}\tilde{\varepsilon}},p) \Big(\norm{F}_{L^{\frac{p}{p-1}}(\Om)}+\norm{|f|}_{(\wpo)^{*}}\Big) \norm{\nabla w_s}_{L^{p}(\Om)}.
 \end{eqnarray*}

We recall  that by approximation and Fatou's lemma \eqref{smalllambda} holds for all $\varphi\in \wpo$. 
 Then by \eqref{smalllambda} we get 
\begin{eqnarray}\label{I4I5}
I_5+ I_6&\leq& (1+{\color{black}\tilde{\varepsilon}})  \mu^{p-1} \lambda \norm{\nabla w_s}^{p}_{L^{p}(\Om)}+\\ 
&& +\    C({\color{black}\tilde{\varepsilon}},p ) \Big(\norm{F}_{L^{\frac{p}{p-1}}(\Om)}+\norm{|f|}_{(\wpo)^{*}}\Big) \norm{\nabla w_s}_{L^{p}(\Om)}.\nonumber
\end{eqnarray}

We now use estimates \eqref{I1}, \eqref{I1prime}, \eqref{I2I3} and \eqref{I4I5} in equality \eqref{I12345} to obtain the following bound
\begin{eqnarray*}
\kappa(\varepsilon)\norm{\nabla w_s}^{p}_{L^{p}(\Om)}&\leq& \frac{1}{\mu}\itl b_0 |\nabla w|^p \chi_{\{|u|> s\}}  dx   +  \frac{1}{\mu}\itl b_1 |u|^{m} e^{p\mu|u|} \chi_{\{|u|> s\}}  dx +\\
&& +\ C({\color{black}\tilde{\varepsilon}}, p)\Big(\norm{F}_{L^{\frac{p}{p-1}}(\Om)}+\norm{|f|}_{(\wpo)^{*}}\Big) \norm{\nabla w_s}_{L^{p}(\Om)},
\end{eqnarray*}
where $\kappa(\varepsilon)=\alpha_0 - (1+\varepsilon)  \mu^{p-1} \lambda$. Observe that when $\delta<\delta_0=(p-1) (\alpha_{0}/\lambda)^{\frac{1}{p-1}}$ we have
$$\mu^{p-1}\lambda= (\de/(p-1))^{p-1}\lambda < (\de_0/(p-1))^{p-1}\lambda = \alpha_0$$ and thus we can choose $\varepsilon>0$ small enough so that $\kappa(\varepsilon)>0$.

Since {$(e^{\mu |u|}-1) \in \wpo$}, by Sobolev's embedding theorem it holds that $e^{p\mu |u|}\in L^{\frac{n}{n-p}}(\Om)$ {if $1<p<n$ and 
$e^{p\mu |u|}\in L^{2}(\Om)$, say, if $p\geq n$}.  Thus we have $|u|^m e^{p\mu |u|} \in L^{1}(\Om)$. Now letting  $s\textcolor{black}{\nearrow}\infty$  in the last \textcolor{black}{inequality,}  we find
\begin{equation*}
 \norm{\nabla w}^{p}_{L^{p}(\Om)} \leq C\Big(\norm{F}_{L^{\frac{p}{p-1}}(\Om)}+\norm{|f|}_{(\wpo)^{*}}\Big) \norm{\nabla w}_{L^{p}(\Om)},
\end{equation*}
which yields 
\begin{equation*}
 \norm{e^{\de |u|/(p-1)}-1}_{\wpo} \leq C(\delta, \lambda, p)\Big(\norm{F}_{L^{\frac{p}{p-1}}(\Om)}+\norm{|f|}_{(\wpo)^{*}}\Big)^{\frac{1}{p-1}}.
\end{equation*}

Finally, note that 
$$\norm{u}_{\wpo}=\norm{\nabla u}_{L^p(\Om)}\leq \frac{p-1}{\de}\norm{\nabla(e^{\de |u|/(p-1)}-1)}_{L^p(\Om)}$$
and \textcolor{black}{hence,}  we also have   
\begin{equation*}
 \norm{u}_{\wpo} \leq C(\delta, \lambda, p)\Big(\norm{F}_{L^{\frac{p}{p-1}}(\Om)}+\norm{|f|}_{(\wpo)^{*}}\Big)^{\frac{1}{p-1}}.
\end{equation*}
This proves inequality \eqref{apri1} for all $\delta\in [\gamma_0, \delta_0)$.

To prove inequality \eqref{apri2} for $\de_{1}$,  we first define $\mu_1=\frac{\de_{1}}{p-1}$ and redefine 
\begin{equation}\label{redefinedw}
 w= {\rm sign}(u) [e^{\mu_1 |u|}-1]/\mu_{1}, \qquad w_s= {\rm sign}(u) [e^{\mu_1 |u_s|}-1]/\mu_{1}.
\end{equation}

Observe that 
$$(e^{\mu_1 |u_s|} -1) e^{\de_1 |u_s|}\geq (1-\varepsilon) e^{ (\delta_1 +\mu_1)|u_s|} -C(\varepsilon,\delta_1) \qquad \text{~for~all~} \varepsilon\in (0,1),$$
and thus  by the first inequality in \eqref{I34B}, with $(\de_1, \mu_1)$ in place of $(\de,\mu)$, we have 
\begin{eqnarray*}
 I_3 + I_4 &\leq &  \itl (-\de_1+\gamma_0) \frac{\al_0}{\mu_1}  (e^{\mu_1 |u_s|} -1) e^{\de_1 |u_s|}  |\nabla u_s|^p \ dx \\
&& +\ \itl \mathcal{B}(x, u, \nabla u) {\rm sign}(u) e^{\de_1 |u_s|} |w_s| \chi_{\{|u|> s\}}  dx\nonumber\\
&\leq& \itl (1-\varepsilon)(-\de_1+\gamma_0) \frac{\al_0}{\mu_1}  |\nabla w_s|^p dx + \itl C(\varepsilon,\delta_1)(\de_1-\gamma_0) \frac{\al_0}{\mu_1}  |\nabla u_s|^p \ dx \nonumber\\
&& +\ \itl \mathcal{B}(x, u, \nabla u) {\rm sign}(u) e^{\de_1 |u_s|} |w_s| \chi_{\{|u|> s\}}  dx.\nonumber
\end{eqnarray*}
Here in the last inequality we used that $\de_1 > \gamma_0$ and $|\nabla w_s|^p= e^{ (\delta_1 +\mu_1)|u_s|} |\nabla u_s|^p$.

Thus arguing as in \eqref{I2I3} for the last term we find
\begin{eqnarray}\label{I34second}
 I_3 + I_4 &\leq& \itl (1-\varepsilon)(-\de_1+\gamma_0) \frac{\al_0}{\mu_1}  |\nabla w_s|^p dx + C(\varepsilon) \itl  |\nabla u_s|^p \ dx \\
&& +\, \frac{1}{\mu_1}\itl b_0 |\nabla w|^p \chi_{\{|u|> s\}}  dx   + \frac{1}{\mu_1}\itl b_1 |u|^{m} e^{p\mu_1 |u|} \chi_{\{|u|> s\}}  dx. \nonumber
\end{eqnarray}

Using estimates \eqref{I1}, \eqref{I1prime},  \eqref{I4I5} (with $(\de_1, \mu_1)$ in place of $(\de,\mu)$) and \eqref{I34second} in equality \eqref{I12345} we then get
\begin{eqnarray*}
\kappa_1(\varepsilon)\norm{\nabla w_s}^{p}_{L^{p}(\Om)}&\leq& \frac{1}{\mu_1}\itl b_0 |\nabla w|^p \chi_{\{|u|> s\}}  dx   +  \frac{1}{\mu_1}\itl b_1 |u|^{m} e^{p\mu_1|u|} \chi_{\{|u|> s\}}  dx +\\
&& +\ C(\varepsilon)\Big(\norm{F}_{L^{\frac{p}{p-1}}(\Om)}+\norm{|f|}_{(\wpo)^{*}}\Big) \norm{\nabla w_s}_{L^{p}(\Om)} +\\
&& +\ C(\varepsilon) \itl  |\nabla u_s|^p \ dx,
\end{eqnarray*}
where $\kappa_1(\varepsilon)=\alpha_0 + (1-\varepsilon)(\de_1-\gamma_0) \frac{\al_0}{\mu_1} - (1+\varepsilon)  \mu_{1}^{p-1} \lambda$, with $\varepsilon\in (0,1)$.  Thus when \eqref{lambdacond2} 
holds we can find $\varepsilon\in (0,1)$ such that $\kappa_1(\varepsilon)>0$. Then using Young's inequality and letting $s\rightarrow\infty$ we eventually obtain
\begin{eqnarray*}
\norm{\nabla w}^{p}_{L^{p}(\Om)}&\leq&  C \Big(\norm{F}_{L^{\frac{p}{p-1}}(\Om)}+\norm{|f|}_{(\wpo)^{*}}\Big)^{\frac{1}{p-1}}  + C  \itl  |\nabla u|^p \ dx.
\end{eqnarray*}

This proves inequality \eqref{apri2} for all $\delta_1 >  \gamma_0$ such \eqref{lambdacond2} holds.
\end{proof}

\section{Existence of solutions to an {approximate} equation} \label{ApproxExist}
{For $k>0$, we now define a function $\mathcal{H}_k(x,s,\xi)$ by letting 
\begin{equation}\label{HK}
\textcolor{black}{\mathcal{H}_k(x,s,\xi):= \frac{\mathcal{B}(x,s,\xi)}{1+ \frac{1}{k}|\, \mathcal{B}(x,s,\xi)| }.}
\end{equation}
Note  $|\mathcal{H}_k(x,s,\xi)|\leq k$, and \eqref{Bcond} also holds with $\mathcal{H}_k(x,s,\xi)$ in place of $\mathcal{B}(x,s,\xi)$.
 Moreover,
$$\lim_{k\rightarrow\infty}\mathcal{H}_k(x,s,\xi)=\mathcal{B}(x,s,\xi).$$}

{
The goal of this section is to obtain existence results for the {approximate} equation
\begin{equation}\label{kequ}
-\dv \aa(x, u, \nabla u) = \mathcal{H}_k(x, u, \nabla u) + \sigma  \trm{in} \Omega.
\end{equation}
}
\begin{proposition}\label{approx-k} Let $\sigma={\rm div}\, F + f$ where \textcolor{black}{$F\in L^{\frac{p}{p-1}}(\Om,\RR^n)$} and $f$ is a locally finite signed measure in $\Om$ with $|f|\in (W^{1,p}_0(\Om))^*$ such that \eqref{smalllambda} holds
for all $\varphi\in C_c^\infty(\Om)$, with   $\lambda\in (0, \gamma_0^{1-p}\alpha_0(p-1)^{p-1}).$
  Then for \textcolor{black}{each $k>0$,}  
there exists a  solution $u_{k}\in W_0^{1,p}(\Om)$ to  \eqref{kequ}  
such that $e^{\frac{\delta|u_k|}{p-1}}-1\in W^{1,p}_0(\Om)$  for all $\de\in [\gamma_0, \de_0)$,  with $\delta_0=(p-1) (\alpha_{0}/\lambda)^{\frac{1}{p-1}}$, and 
\begin{equation}\label{kexp}
 \|u_k\|_{\V} + \|e^{\frac{\delta|u_k|}{p-1}}-1\|_{\V} \leq M_{\de}. 
\end{equation}

Moreover,  for any $\delta_1 > \gamma_0$ such that  \eqref{lambdacond2} holds then we have
\begin{equation}\label{apri2-k}
 \|e^{\frac{\delta_1|u_k|}{p-1}}-1\|_{\V} \leq M_{\de_{1}}, 
\end{equation}
Here the constants $M_\delta$ and $M_{\de_{1}}$  are independent of  $k$.
\end{proposition}

\begin{proof} Since $\sigma\in (\wpo)^{*}$ and  \textcolor{black}{$|{\mathcal{H}}_k(x, s, \xi)|\leq k$, }	 by the theory of pseudomonotone operators (see, e.g., \cite{Li}, \cite[Chapter 6]{MZ}, and \cite{Bre}), for any $\varepsilon>0$ there exists a solution $u_{k, \varepsilon}\in \wpo$ to the equation 
\begin{equation}\label{approx-ep}
-\dv \aa(x, u, \nabla u)  +\varepsilon |u|^{p-2} u= {\mathcal{H}}_k(x, u, \nabla u) + \sigma  \trm{in} \Omega.
\end{equation}

{The next step is to obtain uniform bounds of the form \eqref{kexp}-\eqref{apri2-k} for $\{u_{k, \varepsilon}\}$.
However, we cannot directly apply Theorem   \ref{regularity_Murat-Ferone} here since we do not  know if 
$e^{\frac{\delta|u_{k, \varepsilon}|}{p-1}}-1\in W^{1,p}_0(\Om)$. The strategy here is to follow the proof of Theorem \ref{regularity_Murat-Ferone}. For simplicity let us write $u=u_{k, \varepsilon}$,} 
and  for each $s>0$, we set
$v_s=e^{\de |u_s|} w_s,$
where $u_s= T_s(u)$ and $w_s={\rm sign}(u) [e^{\mu |u_s|} -1]/\mu$ with $\mu=\delta/(p-1)$. Then using $v_s$ as a test function for 
\eqref{approx-ep} we obtain the following equality
\begin{equation}\label{I12345prime}
I_1+I_2=I_3+I_4'+I_5+I_6,
\end{equation}
where the expressions for $I_1, I_2, I_3, I_5, I_6$ are as in the proof of Theorem \ref{regularity_Murat-Ferone}. The term $I_4'$ is similar to $I_4$ given in the 
proof of Theorem \ref{regularity_Murat-Ferone} except that $\mathcal{B}(x, u, \nabla u)$ is now replaced by ${\mathcal{H}}_k(x, u, \nabla u)$. 
That is, 
$$I_4':= \itl {\mathcal{H}}_k(x, u, \nabla u) e^{\de |u_s|} w_s \ dx.$$

Thus lower estimates for $I_1$, $I_2$ and upper estimates for $I_5+I_6$ are unchanged; see \eqref{I1}, \eqref{I1prime},  and \eqref{I4I5}. As in \eqref{I34B}
 we have the following upper estimate for $I_3+I_4'$:
\begin{eqnarray*}
 I_3 + I_4' &\leq &  \itl (-\de+\gamma_0) \al_0 |w_s| e^{\de |u_s|}  |\nabla u_s|^p \ dx \\
&& +\ \itl {\mathcal{H}}_k(x, u, \nabla u) {\rm sign}(u) e^{\de |u_s|} |w_s| \chi_{\{|u|> s\}}  dx\nonumber\\
&\leq& \itl {\mathcal{H}}_k(x, u, \nabla u) {\rm sign}(u) e^{\de |u_s|} |w_s| \chi_{\{|u|> s\}}  dx.\nonumber
\end{eqnarray*}

Thus, instead of \eqref{I2I3}, we now get
\begin{equation}\label{I34prime-k}
I_3 + I_4' \leq  k \itl  e^{\de s} \frac{e^{\mu s}-1}{\mu} \chi_{\{|u|> s\}}  dx. 
\end{equation}

Similarly, instead of \eqref{I34second}, we now obtain
\begin{eqnarray}\label{I34second-k}
 I_3 + I_4' &\leq& \itl (1-\varepsilon)(-\de_1+\gamma_0) \frac{\al_0}{\mu_1}  |\nabla w_s|^p dx + C(\varepsilon) \itl  |\nabla u_s|^p \ dx \\
&& +\,  k \itl  e^{\de s} \frac{e^{\mu s}-1}{\mu} \chi_{\{|u|> s\}}  dx.  \nonumber
\end{eqnarray}
We recall that in \eqref{I34second-k},  $\mu_1=\frac{\de_1}{p-1}$ with  $w, w_s$ to be understood as in \eqref{redefinedw}.

When $\varepsilon>0$ and $s$ is such that $\varepsilon s^{p-1}\geq k$ by  \eqref{I1prime} and  \eqref{I34prime-k} we have 
\begin{equation}\label{ep-absorb}
 I_3 + I_4' -I_2\leq 0
\end{equation}
and thus it follows from \eqref{I12345prime} that 
$$I_1 \leq I_5+ I_6.$$

With this, employing \eqref{I1}  and \eqref{I4I5} we find
\begin{equation*}
\norm{\nabla w_s}_{L^p(\Om)} \leq C(\delta, \lambda,p)\Big(\norm{F}_{L^{\frac{p}{p-1}}(\Om)}+\norm{|f|}_{(\wpo)^{*}}\Big)^{\frac{1}{p-1}}.
\end{equation*}
At this point we let $s\textcolor{black}{\nearrow}\infty$ to obtain that any solution $u=u_{k,\varepsilon}$ to \eqref{approx-ep} satisfies the bound
\textcolor{black}{\begin{equation}\label{exp}
\begin{array}{l}
{\norm{u_{k,\varepsilon}}_{\wpo}+  \norm{e^{\frac{\de |u_{k, \varepsilon}|}{p-1}}-1}_{\wpo}} \leq\\
 \hspace*{3cm} \leq C(\delta, \lambda,p)\Big(\norm{F}_{L^{\frac{p}{p-1}}(\Om)}+\norm{|f|}_{(\wpo)^{*}}\Big)^{\frac{1}{p-1}}
 \end{array}
\end{equation}}
for every $\delta\in [\gamma_0, \delta_0)$.

For $\de_1>\gamma_0$ such that \eqref{lambdacond2} holds, using \eqref{I34second-k} and arguing similarly we obtain
\begin{eqnarray}\label{de1-k}
\norm{e^{\frac{\de_1 |u_{k, \varepsilon}|}{p-1}}-1}_{\wpo}  &\leq&  C \Big(\norm{F}_{L^{\frac{p}{p-1}}(\Om)}+\norm{|f|}_{(\wpo)^{*}}\Big)^{\frac{1}{p-1}}\\
&&  +\,  C  \itl  |\nabla u_{k, \ep}|^p \ dx\nonumber\\
&\leq&  C \Big(\norm{F}_{L^{\frac{p}{p-1}}(\Om)}+\norm{|f|}_{(\wpo)^{*}}\Big)^{\frac{1}{p-1}}. \nonumber
\end{eqnarray}

 As the bound \eqref{exp} is uniform in $\varepsilon$, we can   extract a subsequence, still denoted by $\varepsilon$, such that 
$$u_{k,\varepsilon} \rightarrow u_k \text{~weakly~in~} \wpo,    \text{~strongly~in~} L^p(\Om), \text{~and~a.e.~in~} \Om,$$
as $\varepsilon\textcolor{black}{\searrow} 0^{+}$ for a function $u_k\in \wpo$. 
Due to the pointwise a.e. convergence,  we see that $u_k$ also satisfies \eqref{exp}-\eqref{de1-k} for every $\delta\in [\gamma_0, \delta_0)$ and every $\de_1>\gamma_0$ such
that \eqref{lambdacond2} holds.

Recall that we have 
\begin{equation}\label{ekequn}
-\dv \aa(x, u_{k,\varepsilon},\nabla u_{k,\varepsilon})  +\varepsilon |u_{k,\varepsilon}|^{p-2} u_{k,\varepsilon}= {\mathcal{H}}_k(x, u_{k,\varepsilon}, \nabla u_{k,\varepsilon}) + \sigma  \quad {\rm in~} \mathcal{D}'(\Omega).
\end{equation}

For each fixed $k>0$, we know ${\mathcal{H}}_k(x, u_{k,\varepsilon}, \nabla u_{k,\varepsilon})-\varepsilon |u_{k,\varepsilon}|^{p-2} u_{k,\varepsilon}$ is uniformly bounded in $\varepsilon\in (0,1)$ as finite measures in $\Om$. Thus by a convergence result shown in \textcolor{black}{\cite[Eqn (2.26)]{BM}},
we may further assume that 
$$\nabla u_{k,\varepsilon}\rightarrow \nabla u_k \quad \text{a.e.~in~} \Om, \text{~as~} \varepsilon\textcolor{black}{\searrow} 0^+.$$

This allows us to pass to the limit in \eqref{ekequn} as $\varepsilon\textcolor{black}{\searrow} 0^+$ to see that $u_k$ solves \eqref{kequ} and satisfies the bounds \eqref{kexp}-\eqref{apri2-k}.
\end{proof}

\section{Proof of Theorem \ref{MainExistence}}\label{ActualExistence}

This section is devoted to the proof of Theorem \ref{MainExistence}.

\begin{proof}
For each $k>0$, let $u_k$ be a solution of the {approximate} equation \eqref{kequ} as obtained in Proposition \ref{approx-k}. Recall that 
$\mathcal{H}_k(x, s, \xi)$ is defined in \eqref{HK}.
By \eqref{kexp}-\eqref{apri2-k} and Rellich's compactness theorem, there is a subsequence, still denoted by $k$, such that  
$$u_{k} \xrightarrow{k} u \text{~weakly~in~} \wpo,    \text{~strongly~in~} L^p(\Om), \text{~and~a.e.~in~} \Om,$$
 for some function $ u \in \wpo$ such that $e^{\frac{\delta|u|}{p-1}}-1\in W^{1,p}_0(\Om)$ for each $\de\in [\gamma_0, \de_0)$, and
$e^{\frac{\delta_1|u|}{p-1}}-1\in W^{1,p}_0(\Om)$ for any  $\de_1 >\gamma_0$ such that \eqref{lambdacond2} holds.

As $u_k$ solves \eqref{kequ}, we have 
\begin{equation}\label{kequ-k}
-\dv \aa(x, u_k, \nabla u_k) = {\mathcal{H}}_k(x, u_k, \nabla u_k) + \sigma  \trm{in} \Omega.
\end{equation}

Thus to show that $u$ is a solution of \eqref{basic_pde3} it is enough to show that 
\begin{equation}\label{strongconv}
u_k\rightarrow u \quad \text{strongly~in~} \wpo \text{~as~} k\textcolor{black}{\nearrow}\infty,
\end{equation}
so that we can pass to the limit in \eqref{kequ-k} using \eqref{Bcond}, \eqref{kexp}, and  Vitali's {Convergence} Theorem.

For each $s>0$ we can write
$$\nabla u_k-\nabla u= \nabla T_s(u_k)- \nabla T_s(u) + \nabla G_s(u_k)- \nabla G_s(u),$$
where 
$$G_s(r):=r-T_s(r), \qquad r\in\RR.$$

In order to show \eqref{strongconv} we shall show that the following \textcolor{black}{limits} hold: 
\begin{gather}
\lim_{s\rightarrow\infty}\ \sup_{k>0}\norm{\nabla G_s(u_k)- \nabla G_s(u)}_{L^p(\Om)} =0 \label{tail} \\
\lim_{k\rightarrow\infty}\norm{\nabla T_s(u_k)- \nabla T_s(u)}_{L^p(\Om)} =0 \quad \text{for each } s>0.\label{head}
\end{gather}

The rest of the proof will be devoted to the verification of  these limits.

\noindent {\bf Proof of \eqref{tail}.}
Define $w_k = [e^{\frac{\de}{p-1} |u_k| }-1]\tfrac{p-1}{\delta}$ and hence we get
 \begin{eqnarray*}
  \itl |\nabla G_s(u_k)|^p \ dx&=&\itl[\{|u_k|>s\}] |\nabla u_k|^p \ dx \\
	&=&  \itl[\{|u_k|>s\}] e^{-\frac{\de p}{p-1}|u_k|}|\nabla w_k|^p \ dx \\
  &\leq&  e^{-\frac{\de p}{p-1}s}\itl[\{|u_k|>s\}] |\nabla w_k|^p \ dx.
 \end{eqnarray*}

 Using the  estimate \eqref{kexp}, we then find
\begin{equation}\label{taildelta}
\itl |\nabla G_s(u_k)|^p \ dx\leq C(\delta) e^{-\frac{\de p}{p-1}s},
\end{equation}
  which yields \eqref{tail}.

\noindent {\bf Proof of \eqref{head}.} Following \cite{FM3} (see also the earlier works \cite{FM2, BBM}), we shall make use of the following test function in \eqref{kequ-k}: $$ v_k = \exptk \psi(z_k), \quad {\rm with~} j\geq s,$$ where $z_k = T_s(u_k) - T_s(u)$ and $\psi$ is a $C^1$ and increasing function from $\RR$ to $\RR$ satisfying 
\begin{equation}\label{psicond}
\psi(0) = 0 \quad \text{and} \quad \psi' - H_0 |\psi| \geq 1,
\end{equation}
where $H_0=\frac{b_0+  (a_0+a_1) \de}{\alpha_0}$. For example, the function $\psi(r)= 2re^{\frac{H_0^2 r^2}{4}}$ will do.  We then have 
\begin{eqnarray*}
  \lefteqn{\itl \axgrad{u_k, \nabla u_k} \cdot \exptk \psi'(z_k) \nabla z_k \dx}\\
	&=& \itl \Big[{\mathcal{H}}_k(x, u_k, \nabla u_k)  - \de \mathcal{A}(x,\nabla u_k) \cdot \nabla T_j(u_k) {\rm sign}(u_k)\Big]  \exptk \psi(z_k) \dx  \\
&& + \ \langle\sigma, \exptk \psi(z_k) \rangle.  
\end{eqnarray*}

Note that the term on the left-hand side  in the above equality can be written as
\begin{eqnarray*}
  \lefteqn{\itl \axgrad{u_k, \nabla u_k} \cdot  (\nabla T_s(u_k) - \nabla T_s(u)) \exptk \psi'(z_k)\dx}  \\
& =& \itl[\{\abs{u_k}\leq s\}] (\axgrad{T_s(u_k), \nabla T_s(u_k)} - \axgrad{ T_s(u_k), \nabla T_s (u)} )\cdot \\
&& \qquad \qquad \qquad \cdot\,  (\nabla T_s(u_k) - \nabla T_s(u)) \exptk \psi'(z_k)\dx \\
& &\quad + \itl[\{\abs{u_k}\leq s\}]\axgrad{ T_s(u_k), \nabla T_s (u)} \cdot  (\nabla T_s(u_k) - \nabla T_s(u)) \exptk \psi'(z_k)\dx \\
& &\quad + \itl[\{\abs{u_k}>s\}] \axgrad{ u_k, \nabla u_k} \cdot  (- \nabla T_s(u)) \exptk \psi'(z_k)\dx.
\end{eqnarray*}

Thus combining the last two equalities we obtain
\begin{equation}\label{First5}
I_1=-I_2-I_3 + I_4 +I_5,
\end{equation}
\textcolor{black}{where we have defined}
\begin{eqnarray*}
I_1&=&\itl[\{\abs{u_k}\leq s\}] (\axgrad{ T_s(u_k), \nabla T_s(u_k)} - \axgrad{ T_s(u_k), \nabla T_s (u)} )\cdot\\
&& \qquad \qquad \qquad \cdot\, (\nabla T_s(u_k) - \nabla T_s(u)) \exptk \psi'(z_k)\dx,
\end{eqnarray*}
$$I_2=\itl[\{\abs{u_k}\leq s\}]\axgrad{ T_s(u_k), \nabla T_s (u)} \cdot  (\nabla T_s(u_k) - \nabla T_s(u)) \exptk \psi'(z_k)\dx,$$
$$I_3=\itl[\{\abs{u_k}>s\}] \axgrad{ u_k, \nabla u_k} \cdot  (- \nabla T_s(u)) \exptk \psi'(z_k)\dx,$$
$$I_4=\itl \Big[{\mathcal{H}}_k(x, u_k, \nabla u_k)  - \de \mathcal{A}(x, u_k, \nabla u_k) \cdot \nabla T_j(u_k) {\rm sign}(u_k)\Big]  \exptk \psi(z_k) \dx,$$
and 
$$I_5=\langle\sigma, \exptk \psi(z_k) \rangle.$$

We now write $I_4$ as 
\begin{equation}\label{Secondprime}
I_4= I_4' + I_4'',
\end{equation}
where 
$$I_4':= \itl[\{\abs{u_k}> s\}] H_{k,j}(x) \, \exptk \psi(z_k) \dx,$$
$$I_4'':=\itl[\{\abs{u_k}\leq s\}] H_{k,j}(x) \,  \exptk \psi(z_k) \dx,$$
with $$H_{k,j}(x):= {\mathcal{H}}_k(x, u_k, \nabla u_k)  - \de \mathcal{A}(x, u_k,\nabla u_k) \cdot \nabla T_j(u_k) {\rm sign}(u_k).$$

Note that   $\abs{\nabla T_j(u_k)}\leq  \abs{\nabla u_k}$  and hence using the growth conditions in \eqref{growth-p} and \eqref{Bcond}  we get
\begin{eqnarray*}
|I_4''| &\leq&     \itl[\{\abs{u_k}\leq s\}] \Big(b_0 \abs{\nabla u_k}^p + b_1 |u_k|^m  + \de a_0 \abs{\nabla u_k}^{p} +\delta a_1 |u_k|^{p-1} |\nabla u_k|\Big)\times\\
&& \qquad \qquad \qquad \times   \exptk\  \abs{\psi(z_k)}\    \dx \\
&\leq&  \itl[\{\abs{u_k}\leq s\}] (b_0 + \de (a_0+a_1)  )\abs{\nabla u_k}^p   \exptk\  \abs{\psi(z_k)}\    \dx \\
&& +\  \itl[\{\abs{u_k}\leq s\}] \Big(b_1 |u_k|^m +c(p)\delta a_1 |u_k|^{p} \Big)  \exptk\  \abs{\psi(z_k)}\    \dx\\
&\leq& \frac{b_0 + \de(a_0+a_1)}{\alpha_0}  \itl[\{\abs{u_k}\leq s\}] \mathcal{A}(x, T_s(u_k), \nabla T_s(u_k)) \cdot \nabla T_s(u_k)   \exptk\  \abs{\psi(z_k)}\    \dx \\
&& +\  \itl[\{\abs{u_k}\leq s\}] \Big(b_1 |u_k|^m + c(p)\delta a_1 |u_k|^{p} \Big)  \exptk\  \abs{\psi(z_k)}\    \dx,
\end{eqnarray*}
where we used Young's inequality in the second inequality and  the coercivity condition \eqref{coercivity} in the last inequality. \textcolor{black}{Thus, recalling that 
$H_0=\frac{b_0 + \de(a_0+a_1)}{\alpha_0},$} we find
\begin{eqnarray*}
|I_4''|  
& \leq & H_0 \itl[\{\abs{u_k}\leq s\}]  \left[ \axgrad{ T_s(u_k), \nabla T_s(u_k)} -\axgrad{ T_s(u_k), \nabla T_s(u)} \right] \cdot \\
&& \qquad \qquad \qquad \cdot \left[ \nabla T_s(u_k) - \nabla T_s(u) \right]    \exptk\  \abs{\psi(z_k)}  dx \\
&& + \  H_0\itl[\{\abs{u_k}\leq s\}]\axgrad{ T_s(u_k), \nabla T_s(u)} \cdot \left[ \nabla T_s(u_k) - \nabla T_s(u) \right] \exptk\  \abs{\psi(z_k)}    dx \\
&&  + \   H_0\itl[\{\abs{u_k}\leq s\}]\axgrad{ T_s(u_k), \nabla T_s(u_k)} \cdot \nabla T_s(u)  \exptk\  \abs{\psi(z_k)}   dx \\
&& +\   \itl[\{\abs{u_k}\leq s\}] \Big(b_1 |u_k|^m + c(p)\delta a_1 |u_k|^{p-1} |\nabla T_s(u_k)|\Big)  \exptk\  \abs{\psi(z_k)}   dx.
\end{eqnarray*}
 
Using this bound,  equalities \eqref{First5}-\eqref{Secondprime}, and the inequality  in \eqref{psicond}, we now obtain
\begin{equation}\label{I8}
I_1'\leq -I_2-I_3 + I_4' +I_5 +I_6 + I_7 +I_8,
\end{equation}
where 
$$I_1'=\itl[\{\abs{u_k}\leq s\}] (\axgrad{ T_s(u_k), \nabla T_s(u_k)} - \axgrad{T_s (u), \nabla T_s (u)} )\cdot  (\nabla T_s(u_k) - \nabla T_s(u)) \ dx,$$ 
$$I_6=H_0 \itl[\{\abs{u_k}\leq s\}]\axgrad{T_s(u_k), \nabla T_s(u)} \cdot \left[ \nabla T_s(u_k) - \nabla T_s(u) \right] \exptk\  \abs{\psi(z_k)}\    \dx,$$
$$I_7=H_0 \itl[\{\abs{u_k}\leq s\}]\axgrad{T_s(u_k), \nabla T_s(u_k)} \cdot \nabla T_s(u)  \exptk\  \abs{\psi(z_k)}\    \dx,$$
and
$$I_8=\itl[\{\abs{u_k}\leq s\}] \Big(b_1 |u_k|^m + c(p)\delta a_1 |u_k|^{p-1} |\nabla T_s(u_k)|\Big)  \exptk\  \abs{\psi(z_k)}\    \dx.$$

We shall next treat each term on the right-hand side of \eqref{I8}.

\noindent {\bf The term $I_2$:}  We know that $u_k \xrightarrow{k} u$ a.e., from which we see that $z_k \xrightarrow{k} 0$ a.e. and hence
 $$\axgrad{T_s (u_k), \nabla T_s (u)} \exptk \psi'(z_k) \xrightarrow{k} \axgrad{T_s (u), \nabla T_s (u)} \expt[\abs{T_j(u)}] \psi'(0) \quad {\rm a.e.}$$

Thus using the pointwise estimate, which follows from \eqref{growth-p},
 $$|\axgrad{T_s(u_k), \nabla T_s(u)} e^{\de |T_j(u_k)|} \psi'(z_k)| \leq  e^{\de j} \max_{r \in [-2s,2s]}\abs{\psi'(r)} \Big[ a_0|\nabla T_s(u)|^{p-1} + a_1 s^{p-1}\Big]$$
 and  the fact that \textcolor{black}{$|\nabla T_s(u)|^{p-1} \in L^{\frac{p}{p-1}} (\Om)$}, it follows from Lebesgue's Dominated Convergence Theorem   that 
$$\axgrad{T_s (u_k), \nabla T_s (u)} \exptk \psi'(z_k) \xrightarrow{k} \axgrad{T_s (u), \nabla T_s (u)} \expt[\abs{T_j(u)}] \psi'(0)$$  
strongly in \textcolor{black}{$L^{\frac{p}{p-1}}(\Om,\RR^n)$.  }

Since $ \norm {T_s(u_k)}_{\wpo}$ is uniformly bounded in $k$ and $T_s(u_k) \xrightarrow{k} T_s(u)$ a.e. we get that $\nabla T_s(u_k) \xrightharpoonup{k} \nabla T_s(u)$ weakly in \textcolor{black}{$L^p(\Om,\RR^n)$. }
Also, since 
\begin{equation}\label{chiconv}
\chi_{\{\abs{u_k} \leq s\}} \xrightarrow{k} \chi_{\{\abs{u} \leq s\}}  {\rm~ a.e. ~in~} \Om\setminus\{\abs{u} = s\} {\rm ~while~} |\nabla T_s(u)| = 0 {\rm~ a.e.~ on~} \{\abs{u} = s\},
\end{equation}
 we  have  from Lebesgue's Dominated Convergence Theorem that 
\begin{equation*}
\nabla T_s(u)\chi_{\{\abs{u_k} \leq s\}} \xrightarrow{k} \nabla T_s(u)\chi_{\{\abs{u} \leq s\}}=\nabla T_s(u) \trm{strongly in} \textcolor{black}{L^p(\Om,\RR^n). }
\end{equation*}
Thus with the observation $\chi_{\{\abs{u_k} \leq s\}} (\nabla T_s(u_k) - \nabla T_s(u)) = \nabla T_s(u_k) - \nabla T_s(u)\chi_{\{\abs{u_k} \leq s\}}$,
we see that 
\begin{equation}\label{weaknablaz}
\chi_{\{\abs{u_k} \leq s\}} (\nabla T_s(u_k) - \nabla T_s(u)) \xrightharpoonup{k} 0 \trm{weakly in} \textcolor{black}{L^p(\Om,\RR^n).}
\end{equation}

The above calculations  imply  that $\lim_{k\rightarrow\infty} I_2 =0$.

\noindent {\bf The term $I_3$:} By \eqref{growth-p},  $|\axgrad{ u_k, \nabla u_k}|$ is uniformly bounded in $L^{\frac{p}{p-1}}(\Om)$. On the other hand, again by \eqref{chiconv} and Lebesgue's Dominated Convergence Theorem  we have $$|\chi_{\{\abs{u_k}>s\}}  (-\nabla T_s(u)) \exptk \psi'(z_k)| \xrightarrow{k} 0 \trm{strongly in} L^p(\Om).$$

Thus  we see that $\lim_{k\rightarrow\infty}I_3=0$.

\noindent {\bf The term $I_4'$:} 
We have the inequalities $\axgrad{u_k,  \nabla u_k}\cdot  \nabla T_j(u_k) \geq \alpha_0 \abs{\nabla u_k}^p \chi_{\{\abs{u_k}\leq j\}}$ and  $\chi_{\{\abs{u_k}>s\}} {\rm sign}(u_k) \psi(z_k) \geq 0$. 
Thus using  the second inequality in \eqref{Bcond} we see that
%
\begin{eqnarray*}
 I_4' &=& \itl[\{\abs{u_k}>s\}] \left[ {\rm sign}(u_k) {\mathcal{H}}_k (x,u_k, \nabla u_k) -\de \axgrad{u_k,  \nabla u_k} \cdot \nabla T_j(u_k) \right] \times \\
&& \qquad \qquad \qquad \qquad \times \ {\rm sign}(u_k) 
\exptk \psi(z_k) \dx \\
&\leq& \itl[\{\abs{u_k}>s\}] \left[ \gamma_0\alpha_0 \abs{\nabla u_k}^p -\de \alpha_0\abs{\nabla u_k}^{p} \chi_{\{\abs{u_k}\leq j\}} \right] {\rm sign}(u_k) \exptk \psi(z_k) \dx \\
 & \leq & \itl[\{\abs{u_k}>j\}] \gamma_0\alpha_0 \abs{\nabla u_k}^p  {\rm sign}(u_k) \exptk \psi(z_k) \dx,
\end{eqnarray*}
where we used that $\delta\geq \gamma_0$ {and $j\geq s$} in the last inequality. At this point, using \eqref{taildelta} with $j$ in place of $s$, we get
\begin{eqnarray*}
 I_4'  & \leq & \gamma_0\alpha_0 \max_{r\in[-2s,2s]}|\psi(r)| \  e^{\delta j}\itl[\{\abs{u_k}>j\}]  \abs{\nabla u_k}^p   \dx\\
& \leq & C(\delta) \gamma_0\alpha_0 \max_{r\in[-2s,2s]}|\psi(r)| \  e^{\delta j}\  e^{-\frac{\delta p}{p-1} j}.
\end{eqnarray*}

This yields  that {$\limsup_{j\rightarrow\infty} \sup_{k>0} I_4'=0.$}

\noindent {\bf The term $I_5$:} Since $f\in (W_0^{1,p}(\Om))^*$, there is a vector field \textcolor{black}{$F_1\in L^{\frac{p}{p-1}}(\Om,\RR^n)$} such that ${\rm div}\, F_1 =f$ in $\mathcal{D}'(\Om)$. Thus $\sigma= {\rm div}\, (F+F_1)$ which yields
\begin{eqnarray}\label{I5term}
I_5&=&\de \itl (F+F_1) \cdot \exptk \psi(z_k) \nabla T_j(u_k) {\rm sign}(u_k)\ dx\\
&& +\itl (F+F_1)\cdot \exptk \psi'(z_k) \nabla z_k \ dx.\nonumber
\end{eqnarray}

As $\psi(0)=0$ we have $ (F+F_1)  \exptk \psi(z_k)  \xrightarrow{k} (0, \dots,0)$ a.e. in  $\Om$. Thus by Lebesgue's Dominated Convergence Theorem we find  
 $$ (F+F_1)  \exptk \psi(z_k)   \xrightarrow{k} (0,\dots, 0) \trm{strongly in} \textcolor{black}{L^{\frac{p}{p-1}}(\Om,\RR^n)}.$$
Since $\nabla T_j(u_k) {\rm sign}(u_k)$ is uniformly bounded in $L^p(\Om,\RR^n)$, we then conclude that 
\begin{equation}\label{de1}
\de \itl (F+F_1) \cdot \exptk \psi(z_k) \nabla T_j(u_k) {\rm sign}(u_k)\ dx \xrightarrow{k} 0.
\end{equation}

We now write
\begin{equation}\label{splitF}
\itl (F+F_1)\cdot \exptk \psi'(z_k) \nabla z_k \ dx=  {\color{black}R_1 + R_2}, 
\end{equation}
 {\color{black}where 
 \begin{gather*}
 R_1:= \itl[\{\abs{u_k}\leq s\}](F+F_1)\cdot \exptk \psi'(z_k) \nabla z_k \ dx \\R_2:= \itl[\{\abs{u_k}> s\}](F+F_1)\cdot \exptk \psi'(z_k) \nabla z_k \ dx.
 \end{gather*}}

Again by Lebesgue's Dominated Convergence Theorem we have
$$ (F+F_1) \exptk \psi'(z_k) \xrightarrow{k} {(F+F_1)} \expt[\abs{T_j(u)}] \psi'(0) \trm{strongly in} L^{\frac{p}{p-1}}(\Om,\RR^n).$$ 

Thus using \eqref{weaknablaz} (recall that $\nabla z_k= \nabla T_s(u_k) - \nabla T_s(u)$) we obtain that  
{\color{black}$$R_1\xrightarrow{k} 0.$$}

On the other hand, from the definition of $z_k$ we have 
{\color{black}$$R_2=\itl {(F+F_1)}\cdot \exptk \psi'(z_k) (-\nabla T_s(u)) \chi_{\{\abs{u_k}> s\}}\ dx.$$}

Then by \eqref{chiconv}, H\"older's inequality,  and Lebesgue's Dominated Convergence Theorem, it follows that 
{\color{black}$$R_2 \xrightarrow{k} 0.$$}

Now recalling \eqref{splitF} we get 
\begin{equation}\label{de2}
\itl {(F+F_1)}\cdot \exptk \psi'(z_k) \nabla z_k \ dx \xrightarrow{k} 0.
\end{equation}

Hence using \eqref{de1} and \eqref{de2} in \eqref{I5term} we conclude that 
$\lim_{k\rightarrow\infty} I_5=0$.

\noindent {\bf The terms $I_6$, $I_7$, and $I_8$:}
Since $\psi(0)=0$, by Lebesgue's Dominated Convergence Theorem we find
$$\chi_{\{\abs{u_k}\leq s\}} \axgrad{T_s(u_k), \nabla T_s(u)} \exptk   \abs{\psi(z_k)} \xrightarrow{k} 0 \trm{strongly in } \textcolor{black}{L^{\frac{p}{p-1}}(\Om,\RR^n)}$$
and 
$$\chi_{\{\abs{u_k}\leq s\}} \nabla T_s(u) \exptk   \abs{\psi(z_k)} \xrightarrow{k} 0 \trm{strongly in } \textcolor{black}{L^{p}(\Om,\RR^n).}$$

On the other hand, $\nabla T_s(u_k) - \nabla T_s(u)$ and $\axgrad{T_s(u_k), \nabla T_s(u_k)}$ are uniformly bounded in \textcolor{black}{$L^{p}(\Om,\RR^n)$ and in $L^{\frac{p}{p-1}}(\Om,\RR^n)$}, respectively. Thus we obtain that 
$$\lim_{k\rightarrow\infty} I_6=\lim_{k\rightarrow\infty} I_7=0.$$ 

As for the term $I_8$, we estimate
$$I_8\leq \itl[\{\abs{u_k}\leq s\}]  (b_1 |s|^m + c(p)\delta a_1 s^{p-1})   e^{\delta s}  \abs{\psi(z_k)} dx, $$
which also converges to zero, as $k\textcolor{black}{\nearrow}\infty$, by Lebesgue's Dominated Convergence Theorem.

We have shown that $\lim_{k\rightarrow\infty} (-I_2-I_3 +I_5 +I_6 + I_7 +I_8)=0$ and {$\limsup_{j\rightarrow\infty} \sup_{k>0}I_4'=0$}. 
For each fixed $s>0$, we now let 
$$D_k=(\axgrad{T_s(u_k), \nabla T_s(u_k)} - \axgrad{T_s(u_k), \nabla T_s (u)} )\cdot  (\nabla T_s(u_k) - \nabla T_s(u)).$$
{As $D_k\geq 0$ (by \eqref{monotone-strict}),  in view of \eqref{I8} we find that} 
\begin{equation}\label{Brow1}
\itl[\{\abs{u_k}\leq s\}] D_k \ dx \xrightarrow{k} 0.
\end{equation}

On the other hand, by \eqref{chiconv}, 
\begin{eqnarray*}
\chi_{\{\abs{u_k}> s\}} D_k
&=&\chi_{\{\abs{u_k}> s\}}  [\axgrad{ T_s(u_k), 0} - \axgrad{ T_s(u_k), \nabla T_s (u)}]\cdot  (-\nabla T_s(u))\\
&\rightarrow& 0 \quad \text{a.e. as } k\textcolor{black}{\nearrow}\infty.
\end{eqnarray*}

It then follows from Lebesgue's Dominated Convergence Theorem that
\begin{equation}\label{Brow2}
\itl[\{\abs{u_k} > s\}] D_k \  dx \xrightarrow{k} 0. 
\end{equation}

Combining  \eqref{Brow1}-\eqref{Brow2} we obtain 
\begin{equation*}
\itl[\Om] D_k \  dx \xrightarrow{k} 0.
\end{equation*}

At this point we  use the conditions  \eqref{monotone-strict}-\eqref{growth-p} and a result of  F. E. Browder (see \cite{Bro} or \cite[ Lemma 5]{BMP}) to \textcolor{black}{complete} the proof of \eqref{head}. 
\end{proof}

\section{Proof of Theorems \ref{weakzero} and \ref{Schro-type}} \label{main-proofs}

We are now ready to prove Theorem \ref{weakzero}.
\begin{proof}[Proof of  Theorem \ref{weakzero}] (i) Suppose that \eqref{basic_pde} has a solution in $u\in W^{1,p}_0(\Om)$ such that 
\eqref{nablau-cond} holds for some $A>0$. Then letting $F=|\nabla u|^{p-2}\nabla u$, we immediately have the desired representation for $\sigma$.
 
\noindent (ii) Suppose that  $\sigma={\rm div}\, F + f$ where \textcolor{black}{$F\in L^{\frac{p}{p-1}}(\Om,\RR^n)$} and $f$ is a locally finite signed measure in $\Om$ with $|f|\in (W^{1,p}_0(\Om))^*$ such that  \eqref{datasmallness} holds  for some   $\lambda\in (0, (p-1)^{p-1})$. Applying Theorem \ref{MainExistence} we obtain a solution to \eqref{basic_pde} that satisfies 
all of the properties stated in Theorem \ref{weakzero}(ii) except the Poincar\'e-Sobolev inequality \eqref{nablau-cond}. To verify it, we use $|\varphi|^p$, $\varphi\in C_c^\infty(\Om)$, as a test function in \eqref{basic_pde} to get
$$\int_{\Om} |\varphi|^p|\nabla u|^p dx = p \int_\Om |\nabla u|^{p-2}\nabla u \cdot  \nabla |\varphi| |\varphi|^{p-1} dx + \langle\sigma, |\varphi|^p \rangle.$$

Thus by H\"older's inequality and condition  \eqref{datasmallness} we find
$$\int_{\Om} |\varphi|^p|\nabla u|^p dx \leq  p \left(\int_\Om |\nabla u|^{p} |\varphi|^p dx\right)^{\frac{p-1}{p}}  \left(\int_\Om |\nabla \varphi|^p dx\right)^{\frac{1}{p}} + (p-1)^{p-1} \int_\Om |\nabla \varphi|^p dx.$$

At this point applying Young's inequality we obtain  the Poincar\'e-Sobolev inequality \eqref{nablau-cond} with some $A=A(p)>0$.
\end{proof}

Finally, we prove Theorem \ref{Schro-type}.

\begin{proof}[Proof of Theorem \ref{Schro-type}]  By Theorem \ref{weakzero}(ii) we can find a solution $u\in W^{1,p}_0(\Om) $ to \eqref{basic_pde} such that both 
$e^u-1$ and $e^{\frac{u}{p-1}}-1\in W^{1,p}_0(\Om)$. Thus if we define $v=e^{\frac{u}{p-1}}$ then it holds that $v-1\in W^{1,p}_0(\Om)$ and $v^{p-1}=e^{u}\in W^{1,p}(\Om)$.
We will show that $v$ is indeed a solution of  \eqref{basic_pde-schr}. 

We first observe that the function $e^u |\nabla u|^p$ belongs to $L^1(\Om)$. Indeed,
\begin{eqnarray*}
\int_\Om e^u |\nabla u|^pdx &=&\int_{\{u\geq 0\}\cap\Om}e^u |\nabla u|^pdx + \int_{\{u< 0\}\cap\Om} e^u |\nabla u|^pdx\\
&\leq& \int_{\{u\geq 0\}\cap\Om} e^{p u} |\nabla u|^pdx + \int_{\{u< 0\}\cap\Om} |\nabla u|^pdx\\
&\leq& \int_{\Om}  |\nabla (e^{u})|^pdx + \int_{\Om} |\nabla u|^pdx <+\infty.
\end{eqnarray*}

Let $\varphi\in C_c^\infty(\Om)$. Using $\phi_j:=\varphi \min\{ e^u, j\}$, $j>0$, as a test function for \eqref{basic_pde} we have
\begin{equation}\label{j-test}
\int_{\Om} |\nabla u|^{p-2}\nabla u \cdot \nabla\phi_jdx=\int_{\Om} |\nabla u|^p \phi_j dx + \langle\sigma, \phi_j \rangle.
\end{equation}

We now send $j\textcolor{black}{\nearrow}\infty$  in \eqref{j-test} to obtain   
$$\int_{\Om} |\nabla u|^{p-2}\nabla u \cdot \nabla(\varphi e^u)dx=\int_{\Om} |\nabla u|^p \varphi e^u dx + \langle\sigma, \varphi e^u\rangle.$$
Here we  use $e^u |\nabla u|^p \in L^1(\Om)$ and Lebesgue's Dominated Convergence Theorem. We note that actually by Lemma \ref{dis-mea} we can immediately use $\varphi e^u$ as a test function.  Thus after expanding and simplifying we get
$$\int_{\Om} [|\nabla u|^{p-2}\nabla u\cdot \nabla\varphi ] e^udx= \langle\sigma, \varphi e^u\rangle=\langle\sigma, \varphi v^{p-1}\rangle.$$

Note that  $\nabla v=(p-1)^{-1} e^{\frac{u}{p-1}}\nabla u$ and thus $\nabla u= (p-1) e^{-\frac{u}{p-1}}\nabla v$.
This yields that  
$$(|\nabla u|^{p-2}\nabla u) e^{u}= (p-1)^{p-1} |\nabla v|^{p-2} \nabla v,$$
and hence
$$\int_{\Om} |\nabla v|^{p-2}\nabla v\cdot \nabla\varphi dx= (p-1)^{1-p} \langle\sigma, \varphi v^{p-1}\rangle$$
for all $\varphi\in C_c^\infty(\Om)$. This shows that $v$ is a solution of \eqref{basic_pde-schr} as claimed.

Finally, inequality  \eqref{WNforv}  follows from \eqref{nablau-cond} and the equality $|\frac{\nabla v}{v}|^p=(p-1)^{-p}|\nabla u|^p$.
\end{proof}

\begin{remark}\label{mea-coeff} {The above argument also works for the more general equation 
\begin{equation*}
-{\rm div}\, \mathcal{A}(x, \nabla v) =  (p-1)^{1-p}\, \sigma\, v^{p-1}  \text{ in } \Omega, \qquad v \geq 0  \text{ in } \Omega, \qquad v = 1   \text{ on } \partial \Omega, 
\end{equation*} 
where $\mathcal{A}(x,\xi)$ satisfies \eqref{monotone-strict}-\eqref{growth-p} with $0<\alpha_0\leq a_0$ and the homogeneity condition
$$\mathcal{A}(x, t\xi)= t^{p-1} \mathcal{A}(x, \xi) \qquad  \text{for all  } t>0.$$}
{In this case $v=e^{\frac{u}{p-1}}$, where $u\in W^{1,p}_0(\Om)$ solves the equation 
$$-{\rm div}\, \mathcal{A}(x, \nabla u)= \mathcal{A}(x, \nabla u)\cdot \nabla u + \sigma.$$
By Theorem \ref{MainExistence}, to guarantee that both $e^{u}-1$ and $e^{\frac{u}{p-1}} -1\in W^{1,p}_0(\Om)$, we also need to assume 
$$\lambda\in \Big(0,\  a_0^{1-p} \alpha_0^{p}\, (p-1)^{ p-1}\Big)  \quad \text{if} \quad \frac{a_0}{\alpha_0}\geq p-1$$
and 
$$\lambda\in (0,\  \alpha_0 p-a_0)  \quad \text{if} \quad \frac{a_0}{\alpha_0} < p-1.$$
However, note that no regularity assumption in the $x$-variable of $\mathcal{A}(x,\xi)$ is needed here.}
\end{remark}


\begin{thebibliography}{xx}


\bibitem{ADP}  B. Abdellaoui, A. Dall'Aglio, and I. Peral, {\it Some remarks on elliptic problems with critical growth in the gradient},
J. Differential Equations {\bf 222} (2006) 21--62.


\bibitem{AH} D. R. Adams  and   L. I. Hedberg, {\it Function Spaces and Potential Theory}, Springer-Verlag, Berlin, 1996.



%
%
%

\bibitem{Anc} A. Ancona, {\it On strong barriers and an inequality of Hardy for domains in $\RR^n$},
J.   London   Math.  Soc. {\bf 34} (1986), 274--290.
 
\bibitem{BBM} A. Bensoussan, L. Boccardo, and F. Murat, {\it On a nonlinear partial differential equation having natural growth terms and unbounded solution},
Ann. Inst. H. Poincar\'e Anal. Non Lin\'eaire, {\bf 5} (1988),  347-–364.


\bibitem{BGO} L. Boccardo, T. Gallou\"et, and L. Orsina, {\it Existence and uniqueness of entropy solutions
for nonlinear elliptic equations with measure data}, Ann. Inst. H. Poincar\'e Anal. Non
Lin\'eaire {\bf 13} (1996) 539--551.

\bibitem{BM} L. Boccardo and F. Murat,  {\it Almost everywhere convergence of the gradients of solutions to elliptic and parabolic equations}, Nonlinear Anal. {\bf 19} (1992), 581--597.

 \bibitem{BMP} L.  Boccardo, F. Murat, and  J.-P.  Puel, {\it Existence of bounded solutions for nonlinear elliptic unilateral problems}, Ann. Mat. Pura 
Appl. (4) {\bf 152} (1988), 183--196. 

\bibitem{Bre} H. Br\'ezis, {\it \'Equations et in\'equations non-lin\'eaires dans les espaces vectoriel en dualit\'e}, Ann. Inst. Fourier {\bf 18} (1968), 115--176.

\bibitem{BB} H. Br\'ezis and F.  E.  Browder,  {\it Some properties of higher order Sobolev spaces}, J. Math. Pures Appl.  {\bf 61} (1982),  245--259. 


\bibitem{Bro} F. E. Browder, {\it Existence theorems for nonlinear partial differential equations}, in S.-S. Chern, S. Smale (Eds), Proc. Sympos. Pure Math., Vol. XVI,  pp. 1--60,  Amer. Math. Soc., Providence, R.I. 1970.

%

%
%
%
%
%
%
%
%
%

\bibitem{DPT} D. M. Duc, N. C. Phuc, and T. V. Nguyen, {\it Weighted Sobolev's inequalities for bounded domains and singular elliptic equations}, Indiana Univ. Math. J. {\bf 56} (2007), 615--642.



\bibitem{ChWW} S.-Y.~A. Chang, J.~M. Wilson, and T.~H. Wolff,
{\it Some weighted norm inequalities concerning the Schr\"odinger operators},
Comment. Math. Helv. {\bf 60} (1985), 217--246.

\bibitem{Fef} C. Fefferman, {\it The uncertainty principle}, Bull. Amer. Math. Soc. {\bf 9} (1983), 129--206.


\bibitem{FM1}  V. Ferone and F. Murat, {\it Quasilinear problems having natural growth in the gradient: an existence result when the
source term is small}, in: \'Equations aux d\'eriv\'ees partielles et applications, Articles d\'edi\'es \`a Jacques-Louis Lions, Gauthier-Villars, Paris, 1998, pp. 497--515.

\bibitem{FM2} V. Ferone and F.  Murat, {\it Nonlinear problems having natural growth in the gradient: an existence result when the source terms are small}, Nonlinear Anal., {\bf 42} (2000), 1309--1326. 

\bibitem{FM3} V. Ferone and F. Murat, {\it Nonlinear elliptic equations with natural growth in the gradient and source terms in Lorentz Spaces}, J. Differential Equations {\bf 256} (2014),  577--608.



\bibitem{FV} M.  Frazier and I. E. Verbitsky, {\it Positive solutions to Schr\"odinger's equation and the exponential integrability of the balayage}, Preprint 2015, 	arXiv:1509.09005.

\bibitem{FS} M. Fukushima, K. Sato, and S. Taniguchi, {\it On the closable part of pre-Dirichlet forms and the
fine support of the underlying measures}, Osaka J. Math. {\bf 28} (1991) 517--535.

%
%



\bibitem{HBV} H. A. Hamid and M. F.  Bidaut-Veron, {\it On the connection between two quasilinear elliptic problems with source terms of order $0$ or $1$}. Commun. Contemp. Math. {\bf 12} (2010),  727--788.

\bibitem{HMV} K. Hansson,  V. G. Maz'ya, and I. E. Verbitsky, {\it Criteria of solvability for multidimensional Riccati equations}, Ark. Mat. {\bf 37} (1999),  87-–120.

%
%
%
%


\bibitem{JMV1} B. Jaye, V. G. Maz'ya, and I. E. Verbitsky, {\it Existence and regularity of positive solutions of elliptic equations of Schrödinger type}, J. Anal. Math. {\bf 118} (2012),  577--621.

\bibitem{JMV2} B. Jaye, V. G. Maz'ya, and I. E. Verbitsky, {\it Quasilinear elliptic equations and weighted Sobolev-Poincar\'e inequalities with distributional weights}, Adv. Math. 
{\bf 232} (2013), 513--542. 


\bibitem{KPZ} M. Kardar, G. Parisi, and Y.-C. Zhang, {\it Dynamic scaling of growing interfaces}, Phys. Rev. Lett. {\bf 56} (1986) 889--892.


\bibitem{KS} J. Krug and H. Spohn, {\it Universality classes for deterministic surface growth}, Phys. Rev. A (3) {\bf 38} (1988) 4271--4283.

%


 
%


\bibitem{Lew} J. L. Lewis, {\it Uniformly fat sets}, Trans. Amer. Math. Soc. {\bf 308} (1988),  177--196.


 
%

\bibitem{Li} J.-L. Lions, {\it Quelques m\'ethodes de r\'esolution des probl\`emes aux limites non lin\'eaires}, Dunod; Gauthier-Villars, Paris 1969. xx+554 pp.
 

\bibitem{MZ}  J. Mal\'y and W. P.  Ziemer, {\it Fine regularity of solutions of elliptic partial differential equations}, Mathematical Surveys and Monographs, {\bf 51}. American Mathematical Society, Providence, RI, 1997. xiv+291 pp.

\bibitem{Maz} V. G. Maz'ya, {\it Sobolev Spaces with Applications to Elliptic Partial Differential Equations}, second, revised and augmented ed., in: Grundlehren der math. Wissenschaften, vol. 342, Springer, Heidelberg, 2011.



%
%

\bibitem{MP} T. Mengesha and N. C. Phuc {\it Quasilinear Ricatti type equations with distributional data in Morrey space framework}, J. Differential Equations {\bf 260} (2016),  5421--5449.


%
%
\bibitem{Mik} P. Mikkonen, {\it On the Wolff potential and quasilinear elliptic equations
 involving measures}, Ann. Acad. Sci. Fenn., Ser AI, Math. Dissert. {\bf 104} 1996, 1--71.
%
%
%

\bibitem{P} C. P\'erez, {\it Two weighted inequalities for potential and fractional type maximal operators}, Indiana Univ. Math. J.  {\bf 43}  (1994),  663--683.


%
%

%
%
%
%

\bibitem{SW} E. T. Sawyer and R. L. Wheeden, {\it Weighted inequalities 
for fractional integrals on Euclidean and homogeneous spaces}, Amer. J. Math. 
{\bf 114} (1992), 813--874.

%

%

\end{thebibliography}
\end{document}